\documentclass[11pt]{amsart}
\usepackage[numbers,square]{natbib}
\usepackage{amsmath,amssymb,amsfonts,amsthm}
\usepackage{geometry}
\usepackage{hyperref}
\theoremstyle{plain}
\newtheorem{theorem}{Theorem}
\newtheorem{lemma}{Lemma}
\newtheorem{proposition}{Proposition}
\newtheorem{corollary}{Corollary}
\newtheorem{remark}{Remark}
\numberwithin{equation}{section}
\newcommand{\R}{\mathbb R}
\newcommand{\dd}{\,dx}
\newcommand{\cG}{\mathcal G}
\newcommand{\cE}{\mathcal E}
\allowdisplaybreaks[4]

\begin{document}
	
	\title[Quantitative uniqueness for bi-Laplace equations]{Quantitative uniqueness for bi-Laplace equations with $C^2$ lower-order coefficients}
	\author{Zhiwei Wang}
	\address{Department of Mathematics, University of Science and Technology of China, Anhui Hefei, China}
	
	\begin{abstract}
		We study quantitative unique continuation for bi-Laplace equations
		\[
		\Delta^2u+W(x)\Delta u+V(x)\cdot\nabla u=0
		\]
		with $C^2$ lower-order coefficients. Using the weighted-frequency function, we establish an
		$L^\infty$ three-ball inequality whose exponential factor has
		coefficient dependence $\|W\|_{C^2}^{1/2}+\|V\|_{C^2}^{1/3}$.
		Under some normalization, we obtain a
		vanishing-order bound of size
		$C(\|W\|_{C^2}^{1/2}+\|V\|_{C^2}^{1/3})$.
	\end{abstract}
	
	\maketitle
	
	\section{Introduction}
	\label{sec:introduction}
	
	All of us know a basic complex analysis fact: Analytic function can only possesses isolated zero points. Which directly implies, if an analytic function $u$ vanishes identically in an open subset $\Sigma\subset \Omega$, then $u$ must be trivial in a connected component of $\Omega$ containing $\Sigma$. Also, by Taylor expansion, if we assume for some $x_0 \in \Omega$ and for all $\alpha$, $D^\alpha u(x_0)=0$ holds, then $u$ is trivial. Exactly, we call the first weak unique continuation property (WUCP), and call the second strong unique continuation property (SUCP). If we consider the above property a little bit more, then we can find both of the WUCP and SUCP are false for general smooth functions. For example, the standard mollifier do not possess any of them.
	However, if we are considering $u$ as a solution of liner equation $\mathcal{L}u$=0, then strong unique continuation is known under substantially weaker assumptions
	than smoothness of the coefficients; see, for example,
	\cite{JerisonKenig1985,KochTataru2001}. 
	
	The SUCP is a qualitative property, which means we infer $u=0$ by assuming infinite order of vanishing. Quantitative unique continuation asks for a more precise
	description of this property. Exactly, if a solution is non-trivial, we wonder how large can its vanishing order be, and what is the accurate relation between this bound and the
	size, regularity of the coefficients. For a smooth solution, its order of vanishing at $x_0$ is
	\[
	\min\{\lvert\beta\rvert:D^\beta u(x_0)\neq0\},
	\]
	where $\beta$ is a multi-index; the order is $+\infty$ if every derivative
	of $u$ at $x_0$ vanishes. Equivalently, for the finite vanishing order solution $u$, we write its order at $x_0$ as $\operatorname{ord}_{x_0}u$, and define it to be the maximum of $\alpha$, such that $\limsup_{r\to0}\frac{\|u\|_{L^\infty(B_r(x_0))}}{r^{\alpha}}< +\infty.$
	
	A useful way to obtain a vanishing-order estimate is to establish a
	three-ball inequality. For $0<r_1<r_2<r_3$, such an inequality takes the form
	\begin{equation*}		\|u\|_{L^2(B_{r_2})}
		\leq \exp(\Theta)\|u\|_{L^2(B_{r_1})}^{\theta}
		\|u\|_{L^2(B_{r_3})}^{1-\theta},
		\qquad 0<\theta<1,
	\end{equation*}
	where $L^2$ norm can be replaced by $L^\infty$ norm by using elliptic regularity. Once we obtain such a three-ball inequality, we can quickly reach a lower bound of the form
	\begin{equation*}		\|u\|_{L^\infty(B_r(x_0))}\geq c r^{C\Theta},
		\qquad 0<r<r_0<1,
	\end{equation*}
	which implies $\operatorname{ord}_{x_0}u\leq C\Theta$. We search for the explicit relation between $\Theta$ and the coefficients of equation. There are two standard
	approaches to these estimates: Carleman estimates and frequency functions.
	Carleman estimates apply to many equations with rough coefficients.
	Frequency functions relate unique continuation to the growth of weighted
	integrals of the solution. The frequency function method goes back to
	Almgren and was developed for elliptic equations by Garofalo and Lin
	\cite{GarofaloLin1986}. Weighted frequency functions with a free parameter
	have proved particularly useful for quantitative estimates; see
	\cite{Kukavica2000,Zhu2016,Davey2025}.
	
	The investigation of vanishing order is closely related to volume estimates
	for nodal, critical, and singular sets
	\cite{NaberValtorta2017,HuangJiang2023} and to spectral inequalities
	\cite{ZhuSpectral2026}. A classical quantitative problem
	arises from Laplace eigenfunctions on a smooth closed $n$-dimensional
	Riemannian manifold $(\mathcal M,g)$:
	\begin{equation*}		-\Delta_g\psi_\lambda=\lambda\psi_\lambda.
	\end{equation*}
	Donnelly and Fefferman proved that the vanishing order is bounded by
	$C\lambda^{1/2}$ \cite{DonnellyFefferman1988}. Spherical harmonics show that
	this power of $\lambda$ is sharp. 
	For a nonzero real-valued eigenfunction, write
	$Z_\lambda=\{x\in\mathcal M:\psi_\lambda(x)=0\}$. Yau conjectured
	\cite{Yau1982} that
	\begin{equation}\label{eq:yau-nodal-conjecture}
		c\lambda^{1/2}
		\leq \mathcal H^{n-1}(Z_\lambda)
		\leq C\lambda^{1/2},
	\end{equation}
	where the positive constants depend only on $(\mathcal M,g)$. Donnelly and
	Fefferman proved \eqref{eq:yau-nodal-conjecture} for real-analytic metrics
	\cite{DonnellyFefferman1988}. For smooth manifolds of dimension $n\geq3$,
	Logunov proved that there is an exponent $\alpha_n>1/2$, depending only on
	$n$, such that
	\begin{equation*}		\mathcal H^{n-1}(Z_\lambda)
		\leq C\lambda^{\alpha_n}.
	\end{equation*}
	This polynomial upper bound is obtained through quantitative growth and
	propagation-of-smallness estimates \cite{Logunov2018}. 
	
	For the Schr\"odinger equation
	\begin{equation*}		-\Delta u+q(x)u=0,
	\end{equation*}
	Bourgain and Kenig obtained a vanishing-order bound with coefficient
	dependence $\|q\|_{L^\infty}^{2/3}$ under suitable normalization, in their
	study of the continuous Anderson--Bernoulli model
	\cite{BourgainKenig2005}. Meshkov's example explains the sharpness of this
	power in the complex-valued setting \cite{Meshkov1992}. Additional
	regularity of the potential improves the estimate. For $C^1$ potentials,
	Bakri and Zhu independently obtained the square-root dependence
	$(1+\|q\|_{C^1})^{1/2}$ \cite{Bakri2012,Zhu2016}, in agreement with
	the eigenfunction estimate. For H\"older continuous potentials, the bound
	$(1+\|q\|_{C^{0,\beta}})^{2/(3+\beta)}$, $0<\beta<1$, was obtained by the author and his collaborators in
	\cite{TengWangZhu2026}, which fills the gap and connects the two end point results. Quantitative uniqueness with singular lower order
	terms and near the boundary has also been studied; see
	\cite{DaveyZhu2019,BanerjeeGarofalo2016,LiWang2023}.

	For higher order elliptic equations, the sharp vanishing order result is less complete. The
	strong unique continuation property has been established for powers of
	the Laplacian and products of second order elliptic operators under
	appropriate assumptions; see \cite{Lin2007,ColombiniKoch2010}. It is
	therefore natural to study how the vanishing order depends on the size
	and regularity of the potential. The bi-Laplace operator provides a
	simple model and occurs in plate bending, elasticity, and fluid dynamics.
	For equation
	\begin{equation}\label{eq:intro-bilaplace-potential}
		\Delta^2u+q(x)u=0,
	\end{equation}
	Zhu
	\cite{Zhu2016} used a variant of the frequency function to obtain a
	vanishing-order bound $C(1+\|q\|_{L^\infty})$ in dimensions $n\geq8$.
	Later, using Carleman estimates, Zhu \cite{Zhu2018} obtained the improved
	bound $C(1+\|q\|_{L^\infty}^{1/3})$ for bounded potentials when $n\neq6$.
	In the critical dimension $n=6$, the cited theorem gives
	$C_\varepsilon(1+\|q\|_{L^\infty}^{1/[3(1-\varepsilon)]})$ for every
	sufficiently small $\varepsilon>0$. Using a frequency function approach,
	Liu, Tian, and Yang \cite{LiuTianYang2025} proved that, for
	$W^{1,\infty}$ potentials, the vanishing order of solutions to
	\eqref{eq:intro-bilaplace-potential} is bounded by
	\[
	C\bigl(1+\|q\|_{L^\infty}^{1/4}+\|\nabla q\|_{L^\infty}\bigr).
	\]
	Here the constants depend on the dimension and the normalization of the
	solution. Doubling inequalities and nodal set estimates for bi-Laplace equations
	were studied in \cite{Zhu2019}, and related eigenvalue problems were
	considered in \cite{LinZhu2022}.
	
	In this paper, we study the quantitative unique continuation, or quantitative
	uniqueness, for bi-Laplace equations with lower order terms. We focus on the problem in form
	\begin{equation}\label{eq:intro-equation}
		\Delta^2u+W(x)\Delta u+V(x)\cdot\nabla u=0
		\qquad\text{in }B_R\subset\mathbb R^n,
	\end{equation}
	where $R\geq20$, $n\geq2$, $W$ is a real-valued function, and $V$ is a
	real-valued vector field. We assume that $W$ and the components of $V$ are
	of class $C^2$, and consider real-valued smooth solutions. Our purpose is to
	obtain three-ball inequalities and vanishing-order estimates which describe
	the dependence on the sizes of $W$ and $V$.
	
	Of particular relevance to \eqref{eq:intro-equation} is the work of Fu and
	Liao \cite{FuLiao2021}. They considered bi-Laplace equations containing
	zeroth, first, second, and third order terms, including a scalar multiple
	of $\Delta u$. For bounded coefficients, their estimates give the powers
	$\|W\|_{L^\infty}^{2/3}$ and $\|V\|_{L^\infty}^{1/2}$ for the two terms
	in \eqref{eq:intro-equation}. More recently, the author and his collaborators
	\cite{TengWangZhuBilaplace2026} used weighted frequency functions to
	prove three-ball inequalities and vanishing-order bounds for
	\eqref{eq:intro-bilaplace-potential} with bounded and H\"older continuous
	potentials. They recovered the bound $C(1+\|q\|_{L^\infty}^{1/3})$
	for bounded potentials. Combining the frequency function argument
	with the mollification method in \cite{TengWangZhu2026}, they obtained
	the bound $C(1+\|q\|_{C^{0,\beta}}^{1/4})$ for H\"older continuous potentials.
	
	The present paper concerns the derivative terms $W\Delta u$ and
	$V\cdot\nabla u$. We assume $W,V\in C^2$ and obtain vanishing-order
	estimates like $\|W\|_{C^2}^{1/2}+\|V\|_{C^2}^{1/3}$, under the normalization stated below.
	It would be interesting to combine this argument with mollification,
	as in \cite{TengWangZhu2026,TengWangZhuBilaplace2026}. We leave the
	treatment of H\"older-continuous lower-order coefficients by this approach
	for future work.
	
	We now state the main results. Assume that
	\begin{equation}\label{eq:intro-coefficient-bounds}
		\|W\|_{C^2(B_R)}\leq \mathbb M,\qquad
		\|V\|_{C^2(B_R)}\leq \mathbb G,
		\qquad \mathbb M,\mathbb G\geq1,
	\end{equation}
	and set
	\begin{equation}\label{eq:intro-K}
		\mathbb K:=2+\mathbb M^{1/2}+\mathbb G^{1/3}.
	\end{equation}
	
	\begin{theorem}\label{thm:main-threeball}
		Let $u$ solve \eqref{eq:intro-equation} under
		\eqref{eq:intro-coefficient-bounds}. There exists
		$0<\theta_0<1$, depending on $n,R,\mathbb M,\mathbb G$, such that
		\begin{equation}\label{eq:main-threeball}
			\|u\|_{L^\infty(B_{3r})}
			\leq \exp\bigl(C(\mathbb M^{1/2}+\mathbb G^{1/3})\bigr)
			\|u\|_{L^\infty(B_r)}^{\theta_0}
			\|u\|_{L^\infty(B_{30r})}^{1-\theta_0}
		\end{equation}
		for every $0<r\leq R/40$, where $C$ depends only on $n$ and $R$.
	\end{theorem}
	
	The exponent $\theta_0$ in this statement may depend on the coefficient
	bounds. On the smaller range $0<r\leq\min\{R/40,\mathbb K^{-1}\}$, it can be
	chosen to depend only on $n$ and $R$; see the remark at the end of
	Section~\ref{sec:threeball}.
	
	To study the vanishing order, we assume the normalization
	\begin{equation}\label{eq:normalization-final}
		\|u\|_{L^\infty(B_{10})}\leq\mathcal M,
		\qquad \|u\|_{L^\infty(B_{r_*})}\geq1,
		\qquad \mathcal M\geq1,
	\end{equation}
	where
	\begin{equation*}		r_*:=\frac{1}{100\mathbb K}.
	\end{equation*}
	We have the following quantitative uniqueness result.
	
	\begin{theorem}\label{thm:quantitative-vanishing-order-final}
		Let $u$ solve \eqref{eq:intro-equation} under
		\eqref{eq:intro-coefficient-bounds} and \eqref{eq:normalization-final}.
		Then, for every $x_0\in B_{r_*/2}$,
		\[
		\operatorname{ord}_{x_0}u
		\leq C\bigl(\mathbb M^{1/2}+\mathbb G^{1/3}+\log\mathcal M\bigr),
		\]
		where $C$ depends only on $n$.
	\end{theorem}
	
	Let us explain why the powers $1/2$ and $1/3$ are the natural scales
	in quantitative growth estimates. At a length scale $t^{-1}$, the fourth,
	second, and
	first order terms have respective sizes $t^4$,
	$\|W\|_{L^\infty}t^2$, and $\|V\|_{L^\infty}t$. Balancing the
	lower order terms with the principal part gives
	\begin{equation*}		t^4\sim\|W\|_{L^\infty}t^2
		\quad\Longrightarrow\quad t\sim\|W\|_{L^\infty}^{1/2},
		\qquad
		t^4\sim\|V\|_{L^\infty}t
		\quad\Longrightarrow\quad t\sim\|V\|_{L^\infty}^{1/3}.
	\end{equation*}
	These balances can be seen by simple constant-coefficient examples.
	For $t\geq1$, the real-valued function $u_t(x)=e^{tx_1}$ satisfies
	\begin{equation*}		\Delta^2u_t-t^2\Delta u_t=0,
		\qquad
		\Delta^2u_t-t^3\partial_1u_t=0.
	\end{equation*}
	The only nonzero coefficient has $C^2$ norm $t^2$ in the first
	equation and $t^3$ in the second. On the other hand,
	\begin{equation*}		\log\frac{\|u_t\|_{L^\infty(B_2)}}
		{\|u_t\|_{L^\infty(B_1)}}=t.
	\end{equation*}
	Thus the square-root and cube-root powers recover the logarithmic
	growth and doubling scales of $u_t$ in these two equations.
	
	The paper is organized as follows. Section~\ref{sec:threeball} is devoted
	to the proof of Theorem~\ref{thm:main-threeball}. We establish the
	almost-monotonicity formula, derive the weighted and unweighted three-ball
	inequalities, and obtain the $L^\infty$ estimate. In
	Section~\ref{sec:vanishing}, we prove
	Theorem~\ref{thm:quantitative-vanishing-order-final}. The weighted
	differentiation formula and the interior regularity estimates used in the
	proofs are given in the Appendix. The letters $C,c$
	denote positive constants which may vary from line to line; their
	dependence is specified in the statements. Throughout the paper, $B_r(x_0)$
	denotes the ball centered at $x_0$ with radius $r$, and $B_r=B_r(0)$.
	For the coefficients we use
	$\|f\|_{C^2}=\|f\|_{L^\infty}+\|\nabla f\|_{L^\infty}
	+\|\nabla^2f\|_{L^\infty}$, with the same convention for vector fields.
	
	\section{Three-ball inequality for bi-Laplace equations with lower order terms}
	\label{sec:threeball}
	
	In this section, we prove the quantitative three-ball inequality. We first establish an almost-monotonicity formula for the weighted frequency function. Let $n\geq2$, $R>0$, and let $u$ be a real-valued smooth solution of
	\begin{equation}		\Delta^2u+W(x)\Delta u+V(x)\cdot\nabla u=0
		\qquad\text{in }B_R\subset\R^n .
		\label{eq:main-equation}
	\end{equation}
	Here $W$ and the components of $V=(V^1,\ldots,V^n)$ are real-valued functions satisfying
	\begin{equation*}		W\in C^2(B_R),\qquad V^j\in C^2(B_R),\quad j=1,\ldots,n.
	\end{equation*}
	Put
	\begin{equation*}		w:=\Delta u,
		\qquad
		\omega_r(x):=r^2-|x|^2,
		\qquad
		0<r<R.
	\end{equation*}
	The equation can be written as the second order system $\Delta u=w$ and
	\begin{equation}		\Delta w=-Ww-V\cdot\nabla u .
		\label{eq:Delta-w-equation}
	\end{equation}
	
	We denote the coefficient bounds by
	\begin{align}		W_0&:=\|W\|_{L^\infty(B_R)},
		&W_1&:=\sum_{i=1}^n\|\partial_iW\|_{L^\infty(B_R)},
		&W_2&:=\|\Delta W\|_{L^\infty(B_R)}\notag\\
		V_0&:=\sum_{j=1}^n\|V^j\|_{L^\infty(B_R)},
		&V_1&:=\sum_{i,j=1}^n\|\partial_iV^j\|_{L^\infty(B_R)},
		&V_2&:=\sum_{j=1}^n\|\Delta V^j\|_{L^\infty(B_R)} .\label{eq:V-norms}
	\end{align}
	We assume that these six scalar bounds are finite. Here $V^j$ denotes
	the $j$-th component of $V$, whereas $V_0,V_1,V_2$ denote the scalar
	coefficient bounds in \eqref{eq:V-norms}.
	
	We introduce two parameters
	\begin{equation*}		\alpha\geq2,
		\qquad
		\lambda>0.
	\end{equation*}
	For $0<r<R$, define
	\begin{equation}		\mathbb H(r):=\int_{B_r}\bigl(u^2+\lambda r^2w^2\bigr)\omega_r^{\alpha-1}\dd,
		\label{eq:H-def}
	\end{equation}
	\begin{equation}		\mathbb D(r):=\int_{B_r}\bigl(|\nabla u|^2+\lambda r^2|\nabla w|^2\bigr)\omega_r^\alpha\dd
		+2\lambda\alpha r^2\int_{B_r}w^2\omega_r^{\alpha-1}\dd,
		\label{eq:D-def}
	\end{equation}
	\begin{equation*}		\mathbb L(r):=\int_{B_r}\bigl(u\Delta u+\lambda r^2w\Delta w\bigr)\omega_r^\alpha\dd
		=\int_{B_r}uw\omega_r^\alpha\dd+\lambda r^2\int_{B_r}w\Delta w\,\omega_r^\alpha\dd,
	\end{equation*}
	\begin{equation*}		\mathbb N(r):=\frac{\mathbb D(r)+\mathbb L(r)}{\mathbb H(r)}
	\end{equation*}
	whenever $\mathbb H(r)>0$. The parameters $\alpha$ and $\lambda$ will be chosen below. We first establish the identities and estimates used to study $\mathbb N(r)$.
	
	The weighted differentiation identities used below are stated and proved in Lemma~\ref{lem:weighted-differentiation} in the Appendix. We first estimate the weighted integrals of $\nabla u$ and $D^2u$.
	
	\begin{lemma}
		\label{lem:weighted-derivative-estimates}
		There is a constant $C=C(n)$ such that, for every $0<r<R$,
		\begin{equation}			\int_{B_r}|\nabla u|^2\omega_r^{\alpha+1}\dd
			\leq C\left(\alpha^2r^2+\frac{r^3}{\sqrt\lambda}\right)\mathbb H(r),
			\label{eq:grad-u-control}
		\end{equation}
		and
		\begin{equation}			\int_{B_r}|D^2u|^2\omega_r^{\alpha+2}\dd
			\leq C\alpha^2r^2\mathbb D(r)+C\frac{r^4}{\lambda}\mathbb H(r).
			\label{eq:hessian-u-control}
		\end{equation}
	\end{lemma}
	
	\begin{proof}
		Since $\omega_r^{\alpha+1}=0$ on $\partial B_r$,
		\begin{align}			\int_{B_r}|\nabla u|^2\omega_r^{\alpha+1}\dd
			&=-\int_{B_r}u\Delta u\,\omega_r^{\alpha+1}\dd
			-\int_{B_r}u\nabla u\cdot\nabla(\omega_r^{\alpha+1})\dd\notag\\
			&=-\int_{B_r}uw\omega_r^{\alpha+1}\dd
			-\int_{B_r}u\nabla u\cdot\nabla(\omega_r^{\alpha+1})\dd .
			\label{eq:grad-est-start}
		\end{align}
		For the first term, Cauchy's inequality gives
		\begin{align}			\left|\int_{B_r}uw\omega_r^{\alpha+1}\dd\right|
			&\leq
			\left(\int_{B_r}u^2\omega_r^{\alpha-1}\dd\right)^{1/2}
			\left(\int_{B_r}w^2\omega_r^{\alpha+3}\dd\right)^{1/2}\notag\\
			&\leq r^4
			\left(\int_{B_r}u^2\omega_r^{\alpha-1}\dd\right)^{1/2}
			\left(\int_{B_r}w^2\omega_r^{\alpha-1}\dd\right)^{1/2}\notag\\
			&=\frac{r^3}{\sqrt\lambda}
			\left(\int_{B_r}u^2\omega_r^{\alpha-1}\dd\right)^{1/2}
			\left(\lambda r^2\int_{B_r}w^2\omega_r^{\alpha-1}\dd\right)^{1/2}\notag\\
			&\leq \frac{r^3}{2\sqrt\lambda}\mathbb H(r).
			\label{eq:grad-est-first-term}
		\end{align}
		For the second term, since
		\begin{equation*}			\nabla(\omega_r^{\alpha+1})=-2(\alpha+1)x\omega_r^\alpha,
		\end{equation*}
		we have
		\begin{align*}			-\int_{B_r}u\nabla u\cdot\nabla(\omega_r^{\alpha+1})\dd
			&=2(\alpha+1)\int_{B_r}u(x\cdot\nabla u)\omega_r^\alpha\dd\\
			&=(\alpha+1)\int_{B_r}x\cdot\nabla(u^2)\omega_r^\alpha\dd\\
			&=-(\alpha+1)\int_{B_r}u^2\operatorname{div}(x\omega_r^\alpha)\dd.
		\end{align*}
		Now
		\begin{equation*}			\operatorname{div}(x\omega_r^\alpha)=n\omega_r^\alpha-2\alpha |x|^2\omega_r^{\alpha-1}.
		\end{equation*}
		Since $\omega_r\leq r^2$ and $|x|^2\leq r^2$, it follows that
		\begin{equation*}			\left|\operatorname{div}(x\omega_r^\alpha)\right|
			\leq (n+2\alpha)r^2\omega_r^{\alpha-1}
			\leq C(n)\alpha r^2\omega_r^{\alpha-1}.
		\end{equation*}
		Multiplying by $\alpha+1$ and using $\alpha\geq2$, we get
		\begin{equation}			\left|\int_{B_r}u\nabla u\cdot\nabla(\omega_r^{\alpha+1})\dd\right|
			\leq C(n)\alpha^2r^2\mathbb H(r).
			\label{eq:grad-est-second-term}
		\end{equation}
		Combining \eqref{eq:grad-est-start}, \eqref{eq:grad-est-first-term}, and \eqref{eq:grad-est-second-term} gives \eqref{eq:grad-u-control}.
		
		For the Hessian estimate, put $\phi=\omega_r^{\alpha+2}$.  The Bochner identity gives
		\begin{equation*}			\frac{1}{2}\Delta |\nabla u|^2=|D^2u|^2+\nabla u\cdot\nabla(\Delta u)
			=|D^2u|^2+\nabla u\cdot\nabla w.
		\end{equation*}
		Multiplying by $\phi$ and integrating by parts twice, we obtain
		\begin{align}			\int_{B_r}\phi |D^2u|^2\dd
			&=\frac{1}{2}\int_{B_r}\phi\Delta |\nabla u|^2\dd
			-\int_{B_r}\phi\nabla u\cdot\nabla w\dd\notag\\
			&=\frac{1}{2}\int_{B_r}|\nabla u|^2\Delta\phi\dd
			+\int_{B_r}w\operatorname{div}(\phi\nabla u)\dd\notag\\
			&=\frac{1}{2}\int_{B_r}|\nabla u|^2\Delta\phi\dd
			+\int_{B_r}\phi w^2\dd
			+\int_{B_r}w\nabla\phi\cdot\nabla u\dd .
			\label{eq:hessian-bochner-start}
		\end{align}
		Since $\nabla\omega_r=-2x$,
		\begin{equation*}			\nabla\phi=-2(\alpha+2)x\omega_r^{\alpha+1},
			\qquad
			|\nabla\phi|\leq C\alpha r\omega_r^{\alpha+1}.
		\end{equation*}
		Also
		\begin{align*}			\Delta\phi
			&=-2n(\alpha+2)\omega_r^{\alpha+1}
			+4(\alpha+2)(\alpha+1)|x|^2\omega_r^\alpha,
		\end{align*}
		so
		\begin{equation}			|\Delta\phi|\leq C(n)\alpha^2r^2\omega_r^\alpha,
			\qquad
			\frac{|\nabla\phi|^2}{\phi}
			=4(\alpha+2)^2|x|^2\omega_r^\alpha
			\leq C\alpha^2r^2\omega_r^\alpha.
			\label{eq:phi-bounds}
		\end{equation}
		By Cauchy's inequality,
		\begin{align*}			\left|\int_{B_r}w\nabla\phi\cdot\nabla u\dd\right|
			&\leq \frac{1}{2}\int_{B_r}\phi w^2\dd
			+\frac{1}{2}\int_{B_r}|\nabla u|^2\frac{|\nabla\phi|^2}{\phi}\dd\\
			&\leq \frac{1}{2}\int_{B_r}w^2\omega_r^{\alpha+2}\dd
			+C\alpha^2r^2\int_{B_r}|\nabla u|^2\omega_r^\alpha\dd .
		\end{align*}
		Using this and \eqref{eq:phi-bounds} in \eqref{eq:hessian-bochner-start} yields
		\begin{equation*}			\int_{B_r}|D^2u|^2\omega_r^{\alpha+2}\dd
			\leq C\alpha^2r^2\int_{B_r}|\nabla u|^2\omega_r^\alpha\dd
			+C\int_{B_r}w^2\omega_r^{\alpha+2}\dd .
		\end{equation*}
		The first integral is at most $\mathbb D(r)$, and
		\begin{equation*}			\int_{B_r}w^2\omega_r^{\alpha+2}\dd
			\leq r^6\int_{B_r}w^2\omega_r^{\alpha-1}\dd
			\leq \frac{r^4}{\lambda}\mathbb H(r).
		\end{equation*}
		This proves \eqref{eq:hessian-u-control}.
	\end{proof}
	
	\begin{lemma}
		\label{lem:H-D-prime}
		The following identities hold:
		\begin{equation}			\mathbb H'(r)=\frac{2\alpha+n-2}{r}\mathbb H(r)+\frac{1}{\alpha r}\bigl(\mathbb D(r)+\mathbb L(r)\bigr),
			\label{eq:H-prime}
		\end{equation}
		and, if
		\begin{equation}			X:=x\cdot\nabla u-\frac{\omega_r}{4\alpha}w,
			\qquad
			Y:=x\cdot\nabla w+w-\frac{\omega_r}{4\alpha}\Delta w,
			\label{eq:X-Y-def}
		\end{equation}
		then
		\begin{align}			\mathbb D'(r)
			&=\frac{2\alpha+n-2}{r}\mathbb D(r)
			+\frac{4\alpha}{r}\int_{B_r}\bigl(X^2+\lambda r^2Y^2\bigr)\omega_r^{\alpha-1}\dd\notag\\
			&\quad
			-\frac{1}{4\alpha r}\int_{B_r}\bigl(w^2+\lambda r^2(\Delta w)^2\bigr)\omega_r^{\alpha+1}\dd .
			\label{eq:D-prime}
		\end{align}
	\end{lemma}
	
	\begin{proof}
		Write
		\begin{equation*}			\mathbb H=\mathbb H_u+\lambda r^2\mathbb H_w,
			\qquad
			\mathbb H_u:=\int_{B_r}u^2\omega_r^{\alpha-1}\dd,
			\qquad
			\mathbb H_w:=\int_{B_r}w^2\omega_r^{\alpha-1}\dd .
		\end{equation*}
		By Lemma~\ref{lem:weighted-differentiation} with $\gamma=\alpha-1$ and $f=u^2$,
		\begin{align}			\mathbb H_u'
			&=\frac{2\alpha+n-2}{r}\mathbb H_u
			+\frac{1}{2\alpha r}\int_{B_r}\Delta(u^2)\omega_r^\alpha\dd\notag\\
			&=\frac{2\alpha+n-2}{r}\mathbb H_u
			+\frac{1}{\alpha r}\left(\int_{B_r}|\nabla u|^2\omega_r^\alpha\dd+
			\int_{B_r}uw\omega_r^\alpha\dd\right).
			\label{eq:Hu-prime}
		\end{align}
		Likewise,
		\begin{equation}			\mathbb H_w'=\frac{2\alpha+n-2}{r}\mathbb H_w
			+\frac{1}{\alpha r}\left(\int_{B_r}|\nabla w|^2\omega_r^\alpha\dd+
			\int_{B_r}w\Delta w\,\omega_r^\alpha\dd\right).
			\label{eq:Hw-prime}
		\end{equation}
		Therefore
		\begin{align}			\frac{d}{dr}(\lambda r^2\mathbb H_w)
			&=2\lambda r\mathbb H_w+\lambda r^2\mathbb H_w'\notag\\
			&=\frac{2\alpha+n-2}{r}\lambda r^2\mathbb H_w
			+\frac{1}{\alpha r}\Bigl(\lambda r^2\int_{B_r}|\nabla w|^2\omega_r^\alpha\dd
			+2\lambda\alpha r^2\int_{B_r}w^2\omega_r^{\alpha-1}\dd\notag\\
			&\hspace{42mm}
			+\lambda r^2\int_{B_r}w\Delta w\,\omega_r^\alpha\dd\Bigr).
			\label{eq:lambda-Hw-prime}
		\end{align}
		Adding \eqref{eq:Hu-prime} and \eqref{eq:lambda-Hw-prime} gives \eqref{eq:H-prime}.
		
		We now compute $\mathbb D'$.  Put
		\begin{equation*}			\mathbb D_u:=\int_{B_r}|\nabla u|^2\omega_r^\alpha\dd,
			\qquad
			\widetilde{\mathbb D}_w:=\int_{B_r}|\nabla w|^2\omega_r^\alpha\dd .
		\end{equation*}
		By Lemma~\ref{lem:weighted-differentiation},
		\begin{align}			\mathbb D_u'
			&=\frac{2\alpha+n}{r}\mathbb D_u
			+\frac{1}{r}\int_{B_r}x_j\partial_j(|\nabla u|^2)\omega_r^\alpha\dd\notag\\
			&=\frac{2\alpha+n}{r}\mathbb D_u
			+\frac{2}{r}\int_{B_r}x_j\partial_{ij}u\,\partial_i u\,\omega_r^\alpha\dd,
			\label{eq:Du-start-index}
		\end{align}
		where repeated indices are summed from $1$ to $n$.  Integrating by parts in $x_i$ gives
		\begin{align*}			\int_{B_r}x_j\partial_{ij}u\,\partial_i u\,\omega_r^\alpha\dd
			&=-\int_{B_r}\partial_j u\,\partial_i(\partial_i u\,x_j\omega_r^\alpha)\dd\\
			&=-\int_{B_r}\partial_j u\,\partial_{ii}u\,x_j\omega_r^\alpha\dd
			-\int_{B_r}\partial_j u\,\partial_i u\,\delta_{ij}\omega_r^\alpha\dd\\
			&\quad
			-\int_{B_r}\partial_j u\,\partial_i u\,x_j\partial_i(\omega_r^\alpha)\dd.
		\end{align*}
		Since $\partial_i(\omega_r^\alpha)=-2\alpha x_i\omega_r^{\alpha-1}$, this becomes
		\begin{align}			\int_{B_r}x_j\partial_{ij}u\,\partial_i u\,\omega_r^\alpha\dd
			&=-\int_{B_r}(x\cdot\nabla u)\Delta u\,\omega_r^\alpha\dd
			-\int_{B_r}|\nabla u|^2\omega_r^\alpha\dd\notag\\
			&\quad+2\alpha\int_{B_r}(x\cdot\nabla u)^2\omega_r^{\alpha-1}\dd .
			\label{eq:Du-ibp-index}
		\end{align}
		Using $\Delta u=w$ in \eqref{eq:Du-start-index} and then substituting \eqref{eq:Du-ibp-index}, we get
		\begin{align*}			\mathbb D_u'
			&=\frac{2\alpha+n-2}{r}\mathbb D_u
			+\frac{4\alpha}{r}\int_{B_r}(x\cdot\nabla u)^2\omega_r^{\alpha-1}\dd
			-\frac{2}{r}\int_{B_r}(x\cdot\nabla u)w\omega_r^\alpha\dd.
		\end{align*}
		Now
		\begin{align*}			4\alpha\left(x\cdot\nabla u-\frac{\omega_r}{4\alpha}w\right)^2\omega_r^{\alpha-1}
			&=4\alpha(x\cdot\nabla u)^2\omega_r^{\alpha-1}
			-2(x\cdot\nabla u)w\omega_r^\alpha
			+\frac{1}{4\alpha}w^2\omega_r^{\alpha+1}.
		\end{align*}
		Thus
		\begin{align}			\mathbb D_u'
			&=\frac{2\alpha+n-2}{r}\mathbb D_u
			+\frac{4\alpha}{r}\int_{B_r}\left(x\cdot\nabla u-\frac{\omega_r}{4\alpha}w\right)^2\omega_r^{\alpha-1}\dd
			-\frac{1}{4\alpha r}\int_{B_r}w^2\omega_r^{\alpha+1}\dd.
			\label{eq:Du-prime-final}
		\end{align}
		The same calculation with $w$ in place of $u$ gives
		\begin{align*}			\widetilde{\mathbb D}_w'
			&=\frac{2\alpha+n-2}{r}\widetilde{\mathbb D}_w
			+\frac{4\alpha}{r}\int_{B_r}\left(x\cdot\nabla w-\frac{\omega_r}{4\alpha}\Delta w\right)^2\omega_r^{\alpha-1}\dd\\
			&\quad
			-\frac{1}{4\alpha r}\int_{B_r}(\Delta w)^2\omega_r^{\alpha+1}\dd .
		\end{align*}
		Using \eqref{eq:Hw-prime}, we differentiate the complete $w$-part of $\mathbb D$:
		\begin{align}			&\frac{d}{dr}\left(\lambda r^2\widetilde{\mathbb D}_w+2\lambda\alpha r^2\mathbb H_w\right)\notag\\
			&=2\lambda r\widetilde{\mathbb D}_w+\lambda r^2\widetilde{\mathbb D}_w'
			+4\lambda\alpha r\mathbb H_w+2\lambda\alpha r^2\mathbb H_w'\notag\\
			&=\frac{2\alpha+n-2}{r}\left(\lambda r^2\widetilde{\mathbb D}_w+2\lambda\alpha r^2\mathbb H_w\right)
			+4\lambda\alpha r\int_{B_r}A_w^2\omega_r^{\alpha-1}\dd
			-\frac{\lambda r}{4\alpha}\int_{B_r}(\Delta w)^2\omega_r^{\alpha+1}\dd\notag\\
			&\quad
			+4\lambda r\widetilde{\mathbb D}_w
			+2\lambda r\int_{B_r}w\Delta w\,\omega_r^\alpha\dd
			+4\lambda\alpha r\mathbb H_w,
			\label{eq:w-part-before-square}
		\end{align}
		where
		\begin{equation*}			A_w:=x\cdot\nabla w-\frac{\omega_r}{4\alpha}\Delta w .
		\end{equation*}
		The remaining three terms in \eqref{eq:w-part-before-square} are exactly the cross terms needed to replace $A_w^2$ by $(A_w+w)^2$.  Indeed,
		\begin{align*}			2\alpha\int_{B_r}w(x\cdot\nabla w)\omega_r^{\alpha-1}\dd
			&=-\int_{B_r}w\nabla w\cdot\nabla(\omega_r^\alpha)\dd\\
			&=\int_{B_r}|\nabla w|^2\omega_r^\alpha\dd+\int_{B_r}w\Delta w\,\omega_r^\alpha\dd.
		\end{align*}
		Consequently,
		\begin{align*}			4\alpha\int_{B_r}wA_w\omega_r^{\alpha-1}\dd
			&=4\alpha\int_{B_r}w(x\cdot\nabla w)\omega_r^{\alpha-1}\dd
			-\int_{B_r}w\Delta w\,\omega_r^\alpha\dd\\
			&=2\widetilde{\mathbb D}_w+\int_{B_r}w\Delta w\,\omega_r^\alpha\dd .
		\end{align*}
		Since
		\begin{equation*}			(A_w+w)^2=A_w^2+2wA_w+w^2,
		\end{equation*}
		we have
		\begin{align*}			4\lambda\alpha r\int_{B_r}\bigl((A_w+w)^2-A_w^2\bigr)\omega_r^{\alpha-1}\dd
			&=4\lambda r\widetilde{\mathbb D}_w
			+2\lambda r\int_{B_r}w\Delta w\,\omega_r^\alpha\dd
			+4\lambda\alpha r\mathbb H_w.
		\end{align*}
		Substituting this identity into \eqref{eq:w-part-before-square} gives
		\begin{align}			\frac{d}{dr}\left(\lambda r^2\widetilde{\mathbb D}_w+2\lambda\alpha r^2\mathbb H_w\right)
			&=\frac{2\alpha+n-2}{r}\left(\lambda r^2\widetilde{\mathbb D}_w+2\lambda\alpha r^2\mathbb H_w\right)\notag\\
			&\quad+4\lambda\alpha r\int_{B_r}\left(x\cdot\nabla w+w-\frac{\omega_r}{4\alpha}\Delta w\right)^2\omega_r^{\alpha-1}\dd\notag\\
			&\quad-\frac{\lambda r}{4\alpha}\int_{B_r}(\Delta w)^2\omega_r^{\alpha+1}\dd .
			\label{eq:w-part-final}
		\end{align}
		Adding \eqref{eq:Du-prime-final} and \eqref{eq:w-part-final} proves \eqref{eq:D-prime}.
	\end{proof}
	
	To state the almost-monotonicity formula, let $C_0\geq1$ be a sufficiently
	large dimensional constant, enlarged when necessary in the proofs below. To express the calculation easier, set
	\begin{equation}		B(r):=C_0\left[\frac{r}{\sqrt\lambda}+W_0r^2
		+V_0\left(\alpha\sqrt\lambda\,r^2+\lambda^{1/4}r^{5/2}\right)\right].
		\label{eq:Bstar-def}
	\end{equation}
	Define
	\begin{align}		\cG(r)&:=
		\frac{r}{\sqrt\lambda}
		+\frac{1}{\alpha\sqrt\lambda}
		+\frac{W_0r+W_1r^2}{\alpha}
		+\sqrt\lambda\,V_0r
		+\sqrt\lambda\,V_1r^2
		+\frac{\sqrt\lambda\,V_2r^3}{\alpha}
		+\frac{\lambda V_0^2r^3}{\alpha}
		+\frac{W_0V_0\sqrt\lambda\,r^3}{\alpha},
		\label{eq:G-def}\\
		\cE(r)&:=
		\frac{r}{\alpha\lambda}
		+\frac{W_0r^2}{\alpha\sqrt\lambda}
		+\frac{(W_0^2+W_2)r^3}{\alpha}
		+\frac{W_1r^2+W_1^2r^5}{\alpha}
		+V_0r^2+\frac{V_1r^3}{\alpha}
		+\frac{V_0r^3}{\alpha^2\sqrt\lambda}
		+\cG(r).
		\label{eq:E-def}
	\end{align}
	Finally put
	\begin{equation}		A(r):=C_0\left(\cG(r)+\frac{B(r)}{\alpha r}\right),
		\label{eq:Astar-def}
	\end{equation}
	\begin{equation}		P(r):=C_0\left[\cE(r)+\left(\cG(r)+\frac{2}{r}\right)B(r)+\frac{B(r)^2}{\alpha r}\right].
		\label{eq:Pstar-def}
	\end{equation}
	These quantities record the dependence of the frequency estimates on the coefficients and on the parameters $\alpha$ and $\lambda$.
	
	\begin{lemma}
		\label{lem:elementary-products}
		For every $0<r<R$, the following estimates hold, with $C=C(n)$:
		\begin{align}			\int_{B_r}|\nabla u|\,|\nabla w|\omega_r^{\alpha+1}\dd
			&\leq \frac{r}{\sqrt\lambda}\mathbb D(r),
			\label{eq:prod-a}\\
			\int_{B_r}|w|\,|\nabla w|\omega_r^{\alpha+1}\dd
			&\leq \frac{r}{\lambda}\sqrt{\mathbb H(r)\mathbb D(r)},
			\label{eq:prod-b}\\
			\int_{B_r}|w|\,|\nabla u|\omega_r^{\alpha+1}\dd
			&\leq \frac{r^2}{\sqrt\lambda}\sqrt{\mathbb H(r)\mathbb D(r)},
			\label{eq:prod-c}\\
			\int_{B_r}|w|\,|D^2u|\omega_r^{\alpha+1}\dd
			&\leq C\frac{\alpha r}{\sqrt\lambda}\sqrt{\mathbb H(r)\mathbb D(r)}+C\frac{r^2}{\lambda}\mathbb H(r),
			\label{eq:prod-d}\\
			\int_{B_r}|\nabla w|\,|D^2u|\omega_r^{\alpha+1}\dd
			&\leq C\frac{\alpha}{\sqrt\lambda}\mathbb D(r)+C\frac{r}{\lambda}\sqrt{\mathbb H(r)\mathbb D(r)}.
			\label{eq:prod-e}
		\end{align}
		Moreover,
		\begin{equation}			\int_{B_r}|u|\,|w|\omega_r^{\alpha+1}\dd
			\leq \frac{r^3}{\sqrt\lambda}\mathbb H(r),
			\label{eq:prod-f}
		\end{equation}
		\begin{equation}			\left|\int_{B_r}u(V\cdot\nabla u)\omega_r^{\alpha+1}\dd\right|
			\leq C\left(\alpha V_0r^3+V_1r^4\right)\mathbb H(r),
			\label{eq:prod-g}
		\end{equation}
		and
		\begin{equation}			\lambda r\left|\int_{B_r}w(V\cdot\nabla w)\omega_r^{\alpha+1}\dd\right|
			\leq C\left(\alpha V_0r^2+V_1r^3\right)\mathbb H(r).
			\label{eq:prod-h}
		\end{equation}
	\end{lemma}
	
	\begin{proof}
		Since $0\leq\omega_r\leq r^2$ in $B_r$, we have
		\begin{equation*}			\omega_r^{\alpha+1}\leq r^2\omega_r^\alpha,
			\qquad
			\omega_r^{\alpha+1}\leq r^4\omega_r^{\alpha-1},
			\qquad
			\omega_r^{\alpha+3}\leq r^6\omega_r^\alpha,
			\qquad
			\omega_r^{\alpha+3}\leq r^8\omega_r^{\alpha-1}.
		\end{equation*}
		From \eqref{eq:H-def} and \eqref{eq:D-def},
		\begin{equation*}			\int_{B_r}u^2\omega_r^{\alpha-1}\dd\leq \mathbb H(r),
			\qquad
			\int_{B_r}w^2\omega_r^{\alpha-1}\dd\leq \frac{\mathbb H(r)}{\lambda r^2},
		\end{equation*}
		and
		\begin{equation*}			\int_{B_r}|\nabla u|^2\omega_r^\alpha\dd\leq \mathbb D(r),
			\qquad
			\int_{B_r}|\nabla w|^2\omega_r^\alpha\dd\leq \frac{\mathbb D(r)}{\lambda r^2}.
		\end{equation*}
		Applying Cauchy's inequality gives
		\begin{align*}			\int_{B_r}|\nabla u|\,|\nabla w|\omega_r^{\alpha+1}\dd
			&\leq r^2\int_{B_r}|\nabla u|\,|\nabla w|\omega_r^\alpha\dd\\
			&\leq r^2
			\left(\int_{B_r}|\nabla u|^2\omega_r^\alpha\dd\right)^{1/2}
			\left(\int_{B_r}|\nabla w|^2\omega_r^\alpha\dd\right)^{1/2}\\
			&\leq r^2\mathbb D(r)^{1/2}\left(\frac{\mathbb D(r)}{\lambda r^2}\right)^{1/2}
			=\frac{r}{\sqrt\lambda}\mathbb D(r).
		\end{align*}
		For \eqref{eq:prod-b}, split the weight as
		\begin{equation*}			\omega_r^{\alpha+1}
			=\omega_r^{(\alpha-1)/2}\omega_r^{(\alpha+3)/2}.
		\end{equation*}
		Then
		\begin{align*}			\int_{B_r}|w|\,|\nabla w|\omega_r^{\alpha+1}\dd
			&\leq
			\left(\int_{B_r}w^2\omega_r^{\alpha-1}\dd\right)^{1/2}
			\left(\int_{B_r}|\nabla w|^2\omega_r^{\alpha+3}\dd\right)^{1/2}\\
			&\leq
			\left(\frac{\mathbb H(r)}{\lambda r^2}\right)^{1/2}
			\left(r^6\int_{B_r}|\nabla w|^2\omega_r^\alpha\dd\right)^{1/2}\\
			&\leq
			\left(\frac{\mathbb H(r)}{\lambda r^2}\right)^{1/2}
			\left(r^6\frac{\mathbb D(r)}{\lambda r^2}\right)^{1/2}
			=\frac{r}{\lambda}\sqrt{\mathbb H(r)\mathbb D(r)}.
		\end{align*}
		Similarly,
		\begin{align*}			\int_{B_r}|w|\,|\nabla u|\omega_r^{\alpha+1}\dd
			&\leq
			\left(\int_{B_r}w^2\omega_r^{\alpha-1}\dd\right)^{1/2}
			\left(\int_{B_r}|\nabla u|^2\omega_r^{\alpha+3}\dd\right)^{1/2}\\
			&\leq
			\left(\frac{\mathbb H(r)}{\lambda r^2}\right)^{1/2}
			\left(r^6\int_{B_r}|\nabla u|^2\omega_r^\alpha\dd\right)^{1/2}\\
			&\leq \frac{r^2}{\sqrt\lambda}\sqrt{\mathbb H(r)\mathbb D(r)}.
		\end{align*}
		To prove \eqref{eq:prod-d}, we apply Cauchy's inequality and the Hessian estimate \eqref{eq:hessian-u-control}:
		\begin{align}			\int_{B_r}|w|\,|D^2u|\omega_r^{\alpha+1}\dd
			&\leq
			\left(\int_{B_r}w^2\omega_r^\alpha\dd\right)^{1/2}
			\left(\int_{B_r}|D^2u|^2\omega_r^{\alpha+2}\dd\right)^{1/2}.
			\label{eq:prod-d-start-proof}
		\end{align}
		Since
		\begin{equation*}			\int_{B_r}w^2\omega_r^\alpha\dd
			\leq r^2\int_{B_r}w^2\omega_r^{\alpha-1}\dd
			\leq \frac{\mathbb H(r)}{\lambda},
		\end{equation*}
		\eqref{eq:prod-d-start-proof} gives
		\begin{align*}			\int_{B_r}|w|\,|D^2u|\omega_r^{\alpha+1}\dd
			&\leq \left(\frac{\mathbb H(r)}{\lambda}\right)^{1/2}
			\left(C\alpha^2r^2\mathbb D(r)+C\frac{r^4}{\lambda}\mathbb H(r)\right)^{1/2}\\
			&\leq C\frac{\alpha r}{\sqrt\lambda}\sqrt{\mathbb H(r)\mathbb D(r)}+C\frac{r^2}{\lambda}\mathbb H(r).
		\end{align*}
		For \eqref{eq:prod-e}, use the same Hessian estimate with $|\nabla w|$:
		\begin{align*}			\int_{B_r}|\nabla w|\,|D^2u|\omega_r^{\alpha+1}\dd
			&\leq
			\left(\int_{B_r}|\nabla w|^2\omega_r^\alpha\dd\right)^{1/2}
			\left(\int_{B_r}|D^2u|^2\omega_r^{\alpha+2}\dd\right)^{1/2}\\
			&\leq
			\left(\frac{\mathbb D(r)}{\lambda r^2}\right)^{1/2}
			\left(C\alpha^2r^2\mathbb D(r)+C\frac{r^4}{\lambda}\mathbb H(r)\right)^{1/2}\\
			&\leq C\frac{\alpha}{\sqrt\lambda}\mathbb D(r)+C\frac{r}{\lambda}\sqrt{\mathbb H(r)\mathbb D(r)}.
		\end{align*}
		For \eqref{eq:prod-f}, Cauchy's inequality gives
		\begin{align*}			\int_{B_r}|u|\,|w|\omega_r^{\alpha+1}\dd
			&\leq
			\left(\int_{B_r}u^2\omega_r^{\alpha-1}\dd\right)^{1/2}
			\left(\int_{B_r}w^2\omega_r^{\alpha+3}\dd\right)^{1/2}\\
			&\leq \mathbb H(r)^{1/2}
			\left(r^8\int_{B_r}w^2\omega_r^{\alpha-1}\dd\right)^{1/2}\\
			&\leq \mathbb H(r)^{1/2}\cdot r^4\left(\frac{\mathbb H(r)}{\lambda r^2}\right)^{1/2}
			=\frac{r^3}{\sqrt\lambda}\mathbb H(r).
		\end{align*}
		For \eqref{eq:prod-g}, integration by parts gives
		\begin{align}			\int_{B_r}u(V\cdot\nabla u)\omega_r^{\alpha+1}\dd
			&=\frac{1}{2}\int_{B_r}V\cdot\nabla(u^2)\omega_r^{\alpha+1}\dd\notag\\
			&=-\frac{1}{2}\int_{B_r}(\operatorname{div}V)u^2\omega_r^{\alpha+1}\dd
			-\frac{1}{2}\int_{B_r}u^2V\cdot\nabla(\omega_r^{\alpha+1})\dd.
			\label{eq:prod-g-ibp-proof}
		\end{align}
		The first term in \eqref{eq:prod-g-ibp-proof} is bounded by
		\begin{equation*}			C V_1\int_{B_r}u^2\omega_r^{\alpha+1}\dd
			\leq C V_1r^4\int_{B_r}u^2\omega_r^{\alpha-1}\dd
			\leq C V_1r^4\mathbb H(r).
		\end{equation*}
		The second term uses
		\begin{equation*}			|\nabla(\omega_r^{\alpha+1})|
			=2(\alpha+1)|x|\omega_r^\alpha
			\leq C\alpha r\omega_r^\alpha
			\leq C\alpha r^3\omega_r^{\alpha-1}.
		\end{equation*}
		Hence
		\begin{equation*}			\int_{B_r}u^2|V|\,|\nabla(\omega_r^{\alpha+1})|\dd
			\leq C\alpha V_0r^3\mathbb H(r).
		\end{equation*}
		Together these give \eqref{eq:prod-g}.  Finally, for \eqref{eq:prod-h}, integrate by parts in the same way:
		\begin{align*}			\int_{B_r}w(V\cdot\nabla w)\omega_r^{\alpha+1}\dd
			&=\frac{1}{2}\int_{B_r}V\cdot\nabla(w^2)\omega_r^{\alpha+1}\dd\\
			&=-\frac{1}{2}\int_{B_r}(\operatorname{div}V)w^2\omega_r^{\alpha+1}\dd
			-\frac{1}{2}\int_{B_r}w^2V\cdot\nabla(\omega_r^{\alpha+1})\dd.
		\end{align*}
		The first term satisfies
		\begin{equation*}			\lambda r V_1\int_{B_r}w^2\omega_r^{\alpha+1}\dd
			\leq \lambda r V_1\cdot\frac{r^2}{\lambda}\mathbb H(r)=V_1r^3\mathbb H(r),
		\end{equation*}
		and the second term satisfies
		\begin{align*}			\lambda rV_0\int_{B_r}w^2|\nabla(\omega_r^{\alpha+1})|\dd
			&\leq C\lambda rV_0\alpha r\int_{B_r}w^2\omega_r^\alpha\dd\\
			&\leq C\alpha\lambda r^2V_0\cdot\frac{1}{\lambda}\mathbb H(r)
			=C\alpha V_0r^2\mathbb H(r).
		\end{align*}
		This proves \eqref{eq:prod-h}.
	\end{proof}
	
	\begin{lemma}
		\label{lem:integrand-expansion}
		With $w=\Delta u$ and $\Delta w=-Ww-V\cdot\nabla u$, one has
		\begin{equation}			\Delta(uw)=2\nabla u\cdot\nabla w+w^2-Wuw-u(V\cdot\nabla u),
			\label{eq:expansion-Delta-uw-lemma}
		\end{equation}
		and, using the summation convention,
		\begin{align}			\Delta(w\Delta w)
			&=(\Delta w)^2+W^2w^2+WwV^j\partial_ju-4w\partial_iW\partial_iw-(\Delta W)w^2\notag\\
			&\quad-w(\Delta V^j)\partial_ju-2w\partial_iV^j\partial_{ij}u-wV^j\partial_jw
			-2W|\nabla w|^2\notag\\
			&\quad-2\partial_iw\partial_iV^j\partial_ju-2\partial_iwV^j\partial_{ij}u.
			\label{eq:expansion-wDeltaw-lemma}
		\end{align}
	\end{lemma}
	
	\begin{proof}
		Since $w=\Delta u$, the product rule gives
		\begin{align*}			\Delta(uw)
			&=u\Delta w+w\Delta u+2\nabla u\cdot\nabla w\\
			&=u\Delta w+w^2+2\nabla u\cdot\nabla w.
		\end{align*}
		Using $\Delta w=-Ww-V\cdot\nabla u$, this becomes
		\begin{equation*}			\Delta(uw)=-Wuw-u(V\cdot\nabla u)+w^2+2\nabla u\cdot\nabla w,
		\end{equation*}
		which is \eqref{eq:expansion-Delta-uw-lemma}.
		
		To prove \eqref{eq:expansion-wDeltaw-lemma}, we write
		\begin{equation}			\Delta(w\Delta w)=(\Delta w)^2+w\Delta(\Delta w)+2\nabla w\cdot\nabla(\Delta w).
			\label{eq:Delta-wDeltaw-product-start}
		\end{equation}
		The first derivatives of $\Delta w$ are
		\begin{align}			\partial_i(\Delta w)
			&=\partial_i(-Ww-V^j\partial_j u)\notag\\
			&=-(\partial_iW)w-W\partial_iw-(\partial_iV^j)\partial_ju-V^j\partial_{ij}u.
			\label{eq:partial-Deltaw-expansion-lemma}
		\end{align}
		Next compute $\Delta(\Delta w)$.  Since
		\begin{equation*}			\Delta(\Delta w)=-\Delta(Ww)-\Delta(V^j\partial_j u),
		\end{equation*}
		we compute
		\begin{align}			\Delta(Ww)
			&=\partial_i\partial_i(Ww)\notag\\
			&=\partial_i\bigl((\partial_iW)w+W\partial_iw\bigr)\notag\\
			&=(\partial_{ii}W)w+(\partial_iW)\partial_iw+(\partial_iW)\partial_iw+W\partial_{ii}w\notag\\
			&=(\Delta W)w+2\partial_iW\partial_iw+W\Delta w\notag\\
			&=(\Delta W)w+2\partial_iW\partial_iw+W(-Ww-V^j\partial_ju)\notag\\
			&=(\Delta W)w+2\partial_iW\partial_iw-W^2w-WV^j\partial_ju.
			\label{eq:Delta-Ww-expansion-lemma}
		\end{align}
		Similarly,
		\begin{align}			\Delta(V^j\partial_ju)
			&=\partial_i\partial_i(V^j\partial_ju)\notag\\
			&=\partial_i\bigl((\partial_iV^j)\partial_ju+V^j\partial_{ij}u\bigr)\notag\\
			&=(\partial_{ii}V^j)\partial_ju+(\partial_iV^j)\partial_{ij}u
			+(\partial_iV^j)\partial_{ij}u+V^j\partial_{iij}u\notag\\
			&=(\Delta V^j)\partial_ju+2\partial_iV^j\partial_{ij}u+V^j\partial_j(\Delta u)\notag\\
			&=(\Delta V^j)\partial_ju+2\partial_iV^j\partial_{ij}u+V^j\partial_jw.
			\label{eq:Delta-Vgradu-expansion-lemma}
		\end{align}
		Combining \eqref{eq:Delta-Ww-expansion-lemma} and \eqref{eq:Delta-Vgradu-expansion-lemma},
		\begin{align}			\Delta(\Delta w)
			&=W^2w+WV^j\partial_ju-2\partial_iW\partial_iw-(\Delta W)w\notag\\
			&\quad-(\Delta V^j)\partial_ju-2\partial_iV^j\partial_{ij}u-V^j\partial_jw.
			\label{eq:Delta-Deltaw-expansion-lemma}
		\end{align}
		Now substitute \eqref{eq:partial-Deltaw-expansion-lemma} and \eqref{eq:Delta-Deltaw-expansion-lemma} into \eqref{eq:Delta-wDeltaw-product-start}.  First,
		\begin{align}			w\Delta(\Delta w)
			&=W^2w^2+WwV^j\partial_ju-2w\partial_iW\partial_iw-(\Delta W)w^2\notag\\
			&\quad-w(\Delta V^j)\partial_ju-2w\partial_iV^j\partial_{ij}u-wV^j\partial_jw.
			\label{eq:w-Delta-Deltaw-expansion-lemma}
		\end{align}
		Second,
		\begin{align}			2\nabla w\cdot\nabla(\Delta w)
			&=2\partial_iw\partial_i(\Delta w)\notag\\
			&=-2\partial_iw(\partial_iW)w-2W|\nabla w|^2
			-2\partial_iw(\partial_iV^j)\partial_ju-2\partial_iwV^j\partial_{ij}u.
			\label{eq:gradw-gradDeltaw-expansion-lemma}
		\end{align}
		Adding \eqref{eq:w-Delta-Deltaw-expansion-lemma} and \eqref{eq:gradw-gradDeltaw-expansion-lemma}, we obtain
		\begin{align*}			w\Delta(\Delta w)+2\nabla w\cdot\nabla(\Delta w)
			&=W^2w^2+WwV^j\partial_ju-4w\partial_iW\partial_iw-(\Delta W)w^2\\
			&\quad-w(\Delta V^j)\partial_ju-2w\partial_iV^j\partial_{ij}u-wV^j\partial_jw\\
			&\quad-2W|\nabla w|^2-2\partial_iw\partial_iV^j\partial_ju-2\partial_iwV^j\partial_{ij}u.
		\end{align*}
		Adding the remaining first term $(\Delta w)^2$ in \eqref{eq:Delta-wDeltaw-product-start} proves \eqref{eq:expansion-wDeltaw-lemma}.
	\end{proof}
	
	\begin{lemma}
		\label{lem:true-frequency-lower-bound}
		For every $0<r<R$ such that $\mathbb H(r)>0$, one has
		\begin{equation}			\mathbb N(r)\geq -B(r),
			\qquad
			|\mathbb N(r)|\leq \mathbb N(r)+2B(r),
			\label{eq:true-frequency-lower-bound}
		\end{equation}
	\end{lemma}
	
	\begin{proof}
		By \eqref{eq:Delta-w-equation},
		\begin{align}			\mathbb L(r)
			&=\int_{B_r}uw\omega_r^\alpha\dd
			-\lambda r^2\int_{B_r}Ww^2\omega_r^\alpha\dd
			-\lambda r^2\int_{B_r}w(V\cdot\nabla u)\omega_r^\alpha\dd.
			\label{eq:L-split-main}
		\end{align}
		For the first term,
		\begin{align}			\left|\int_{B_r}uw\omega_r^\alpha\dd\right|
			&\leq
			\left(\int_{B_r}u^2\omega_r^{\alpha-1}\dd\right)^{1/2}
			\left(\int_{B_r}w^2\omega_r^{\alpha+1}\dd\right)^{1/2}\notag\\
			&\leq r^2
			\left(\int_{B_r}u^2\omega_r^{\alpha-1}\dd\right)^{1/2}
			\left(\int_{B_r}w^2\omega_r^{\alpha-1}\dd\right)^{1/2}\notag\\
			&=\frac{r}{\sqrt\lambda}
			\left(\int_{B_r}u^2\omega_r^{\alpha-1}\dd\right)^{1/2}
			\left(\lambda r^2\int_{B_r}w^2\omega_r^{\alpha-1}\dd\right)^{1/2}\notag\\
			&\leq \frac{r}{2\sqrt\lambda}\mathbb H(r).
			\label{eq:L-uw-est}
		\end{align}
		For the $W$-term,
		\begin{align}			\lambda r^2\left|\int_{B_r}Ww^2\omega_r^\alpha\dd\right|
			&\leq \lambda r^2W_0r^2\int_{B_r}w^2\omega_r^{\alpha-1}\dd
			\leq W_0r^2\mathbb H(r).
			\label{eq:L-W-est}
		\end{align}
		For the drift term, use Cauchy's inequality and \eqref{eq:grad-u-control}:
		\begin{align}			&\lambda r^2\left|\int_{B_r}w(V\cdot\nabla u)\omega_r^\alpha\dd\right|\notag\\
			&\quad\leq \lambda r^2V_0
			\left(\int_{B_r}w^2\omega_r^{\alpha-1}\dd\right)^{1/2}
			\left(\int_{B_r}|\nabla u|^2\omega_r^{\alpha+1}\dd\right)^{1/2}\notag\\
			&\quad\leq \lambda r^2V_0\frac{\mathbb H(r)^{1/2}}{\sqrt\lambda r}
			C\left(\alpha r+\frac{r^{3/2}}{\lambda^{1/4}}\right)\mathbb H(r)^{1/2}\notag\\
			&\quad\leq C V_0\left(\alpha\sqrt\lambda\,r^2+\lambda^{1/4}r^{5/2}\right)\mathbb H(r).
			\label{eq:L-V-est}
		\end{align}
		Combining \eqref{eq:L-split-main}--\eqref{eq:L-V-est}, and enlarging
		$C_0$ in \eqref{eq:Bstar-def} if necessary, gives
		\begin{equation}			|\mathbb L(r)|\leq \frac14 B(r)\mathbb H(r).
			\label{eq:L-by-BH}
		\end{equation}
		Since $\mathbb D(r)\geq0$,
		\begin{equation*}			\mathbb N(r)=\frac{\mathbb D(r)+\mathbb L(r)}{\mathbb H(r)}
			\geq -\frac{|\mathbb L(r)|}{\mathbb H(r)}\geq-\frac14B(r)\geq-B(r).
		\end{equation*}
		If $\mathbb N(r)\geq0$, then $|\mathbb N(r)|=\mathbb N(r)\leq \mathbb N(r)+2B(r)$.  If $\mathbb N(r)<0$, the lower bound gives $-\mathbb N(r)\leq B(r)$, and hence
		\begin{equation*}			|\mathbb N(r)|=-\mathbb N(r)=\mathbb N(r)-2\mathbb N(r)\leq \mathbb N(r)+2B(r).
		\end{equation*}
		Moreover, \eqref{eq:L-by-BH} gives the stronger estimates
		\begin{equation}			\mathbb N(r)\geq-\frac14B(r),\qquad
			\frac{\mathbb D(r)}{\mathbb H(r)}\leq \mathbb N(r)+\frac14B(r),\qquad
			|\mathbb N(r)|\leq \mathbb N(r)+\frac12B(r).
			\label{eq:true-frequency-strong-bounds}
		\end{equation}
		Thus \eqref{eq:true-frequency-lower-bound} is proved.
	\end{proof}
	
	\begin{proposition}
		\label{prop:true-frequency-differential-ineq}
		For every $0<r<R$ such that $\mathbb H(r)>0$, one has
		\begin{equation}			\mathbb N'(r)+A(r)\mathbb N(r)\geq -P(r).
			\label{eq:true-frequency-differential-ineq}
		\end{equation}
	\end{proposition}
	
	\begin{proof}
		We first estimate the contribution of $\mathbb D$ to $\mathbb N'$. Let $X$ and $Y$ be
		as in \eqref{eq:X-Y-def}. Integration by parts gives
		\begin{align}			2\alpha\int_{B_r}uX\omega_r^{\alpha-1}\dd
			&=2\alpha\int_{B_r}u(x\cdot\nabla u)\omega_r^{\alpha-1}\dd
			-\frac{1}{2}\int_{B_r}uw\omega_r^\alpha\dd\notag\\
			&=\int_{B_r}|\nabla u|^2\omega_r^\alpha\dd+
			\frac{1}{2}\int_{B_r}uw\omega_r^\alpha\dd.
			\label{eq:uX-ident}
		\end{align}
		Here we used $\nabla(\omega_r^\alpha)=-2\alpha x\omega_r^{\alpha-1}$ and $\Delta u=w$. Similarly,
		\begin{align}			2\alpha\lambda r^2\int_{B_r}wY\omega_r^{\alpha-1}\dd
			&=\lambda r^2\int_{B_r}|\nabla w|^2\omega_r^\alpha\dd
			+2\lambda\alpha r^2\int_{B_r}w^2\omega_r^{\alpha-1}\dd\notag\\
			&\quad+
			\frac{1}{2}\lambda r^2\int_{B_r}w\Delta w\,\omega_r^\alpha\dd.
			\label{eq:wY-ident}
		\end{align}
		Adding \eqref{eq:uX-ident} and \eqref{eq:wY-ident} gives
		\begin{equation}			2\alpha\int_{B_r}(uX+\lambda r^2wY)\omega_r^{\alpha-1}\dd
			=\mathbb D(r)+\frac{1}{2}\mathbb L(r).
			\label{eq:J-ident-main}
		\end{equation}
		By Cauchy's inequality applied to the two vector-valued functions $(u,\sqrt\lambda rw)$ and $(X,\sqrt\lambda rY)$ with measure $\omega_r^{\alpha-1}dx$,
		\begin{align*}			&\left(\int_{B_r}(uX+\lambda r^2wY)\omega_r^{\alpha-1}\dd\right)^2\\
			&\quad\leq
			\left(\int_{B_r}(u^2+\lambda r^2w^2)\omega_r^{\alpha-1}\dd\right)
			\left(\int_{B_r}(X^2+\lambda r^2Y^2)\omega_r^{\alpha-1}\dd\right)\\
			&\quad=\mathbb H(r)\int_{B_r}(X^2+\lambda r^2Y^2)\omega_r^{\alpha-1}\dd .
		\end{align*}
		Using \eqref{eq:J-ident-main}, this becomes
		\begin{equation}			\mathbb H(r)\int_{B_r}(X^2+\lambda r^2Y^2)\omega_r^{\alpha-1}\dd
			\geq \frac{1}{4\alpha^2}\left(\mathbb D(r)+\frac{1}{2}\mathbb L(r)\right)^2.
			\label{eq:CS-D-part-final}
		\end{equation}
		From \eqref{eq:H-prime} and \eqref{eq:D-prime}, the common $(2\alpha+n-2)/r$ terms cancel in $\mathbb D'\mathbb H-\mathbb D\mathbb H'$, and hence
		\begin{align*}			\mathbb D'\mathbb H-\mathbb D\mathbb H'
			&=\frac{4\alpha}{r}\mathbb H\int_{B_r}(X^2+\lambda r^2Y^2)\omega_r^{\alpha-1}\dd\\
			&\quad-
			\frac{\mathbb H}{4\alpha r}\int_{B_r}\bigl(w^2+\lambda r^2(\Delta w)^2\bigr)\omega_r^{\alpha+1}\dd
			-\frac{\mathbb D(\mathbb D+\mathbb L)}{\alpha r}.
		\end{align*}
		By \eqref{eq:CS-D-part-final},
		\begin{align*}			&\frac{4\alpha}{r}\mathbb H\int_{B_r}(X^2+\lambda r^2Y^2)\omega_r^{\alpha-1}\dd
			-\frac{\mathbb D(\mathbb D+\mathbb L)}{\alpha r}\\
			&\quad\geq \frac{1}{\alpha r}\left(\mathbb D+\frac{1}{2}\mathbb L\right)^2-\frac{\mathbb D(\mathbb D+\mathbb L)}{\alpha r}
			=\frac{\mathbb L^2}{4\alpha r}\geq0.
		\end{align*}
		Therefore
		\begin{equation}			\frac{\mathbb D'\mathbb H-\mathbb D\mathbb H'}{\mathbb H^2}
			\geq -\frac{1}{4\alpha r\mathbb H}\int_{B_r}\bigl(w^2+\lambda r^2(\Delta w)^2\bigr)\omega_r^{\alpha+1}\dd.
			\label{eq:D-part-negative}
		\end{equation}
		The first term satisfies
		\begin{equation}			\int_{B_r}w^2\omega_r^{\alpha+1}\dd
			\leq r^4\int_{B_r}w^2\omega_r^{\alpha-1}\dd
			\leq \frac{r^2}{\lambda}\mathbb H(r).
			\label{eq:w2-alpha1-H}
		\end{equation}
		For the second term, \eqref{eq:Delta-w-equation} gives
		\begin{equation}			(\Delta w)^2=(-Ww-V\cdot\nabla u)^2
			\leq 2W_0^2w^2+2V_0^2|\nabla u|^2.
			\label{eq:Deltaw-square-bound}
		\end{equation}
		Thus
		\begin{align}			\lambda r^2\int_{B_r}(\Delta w)^2\omega_r^{\alpha+1}\dd
			&\leq C\lambda r^2W_0^2\int_{B_r}w^2\omega_r^{\alpha+1}\dd
			+C\lambda r^2V_0^2\int_{B_r}|\nabla u|^2\omega_r^{\alpha+1}\dd\notag\\
			&\leq CW_0^2r^4\mathbb H(r)+C\lambda V_0^2r^4\mathbb D(r),
			\label{eq:Deltaw-square-integral-D}
		\end{align}
		because $\omega_r^{\alpha+1}\leq r^2\omega_r^\alpha$.  Also
		\begin{equation}			\frac{\mathbb D(r)}{\mathbb H(r)}=\mathbb N(r)-\frac{\mathbb L(r)}{\mathbb H(r)}\leq \mathbb N(r)+B(r).
			\label{eq:D-over-H-upper-main}
		\end{equation}
		Substituting \eqref{eq:w2-alpha1-H} and \eqref{eq:Deltaw-square-integral-D} into \eqref{eq:D-part-negative}, and then using \eqref{eq:D-over-H-upper-main}, gives
		\begin{align}			\frac{\mathbb D'\mathbb H-\mathbb D\mathbb H'}{\mathbb H^2}
			&\geq -C\left(\frac{r}{\alpha\lambda}+\frac{W_0^2r^3}{\alpha}
			+\frac{\lambda V_0^2r^3}{\alpha}\bigl(\mathbb N(r)+B(r)\bigr)\right).
			\label{eq:D-part-final}
		\end{align}
		
		To estimate the contribution of $\mathbb L$, write
		\begin{equation*}			\mathbb L=\mathbb L_1+\mathbb L_2,
			\qquad
			\mathbb L_1:=\int_{B_r}uw\omega_r^\alpha\dd,
			\qquad
			\mathbb L_2:=\lambda r^2\int_{B_r}w\Delta w\,\omega_r^\alpha\dd .
		\end{equation*}
		For $\mathbb L_1$, Lemma~\ref{lem:weighted-differentiation} with $f=uw$ and $\gamma=\alpha$ gives
		\begin{equation}			\mathbb L_1'=\frac{2\alpha+n}{r}\mathbb L_1+
			\frac{1}{2(\alpha+1)r}\int_{B_r}\Delta(uw)\omega_r^{\alpha+1}\dd.
			\label{eq:L1-prime-main}
		\end{equation}
		Because $\Delta u=w$ and $\Delta w=-Ww-V\cdot\nabla u$,
		\begin{align}			\Delta(uw)
			&=u\Delta w+w\Delta u+2\nabla u\cdot\nabla w\notag\\
			&=-Wuw-u(V\cdot\nabla u)+w^2+2\nabla u\cdot\nabla w.
			\label{eq:Delta-uw-main}
		\end{align}
		We estimate the terms in \eqref{eq:Delta-uw-main} by Lemma~\ref{lem:elementary-products}. From \eqref{eq:prod-a},
		\begin{equation}			\frac{1}{2(\alpha+1)r}\left|\int_{B_r}2\nabla u\cdot\nabla w\,\omega_r^{\alpha+1}\dd\right|
			\leq \frac{C}{\alpha\sqrt\lambda}\mathbb D(r).
			\label{eq:L1-grad-inline}
		\end{equation}
		The term containing $w^2$ is nonnegative. In particular, \eqref{eq:w2-alpha1-H} gives
		\begin{equation}			0\leq\frac{1}{2(\alpha+1)r}\int_{B_r}w^2\omega_r^{\alpha+1}\dd
			\leq\frac{Cr}{\alpha\lambda}\mathbb H(r).
			\label{eq:L1-w2-inline}
		\end{equation}
		Furthermore, \eqref{eq:prod-f} and \eqref{eq:prod-g} imply
		\begin{equation}			\frac{1}{2(\alpha+1)r}\left|\int_{B_r}Wuw\omega_r^{\alpha+1}\dd\right|
			\leq \frac{CW_0r^2}{\alpha\sqrt\lambda}\mathbb H(r),
			\label{eq:L1-Wuw-inline}
		\end{equation}
		and
		\begin{equation}			\frac{1}{2(\alpha+1)r}\left|\int_{B_r}u(V\cdot\nabla u)\omega_r^{\alpha+1}\dd\right|
			\leq C\left(V_0r^2+\frac{V_1r^3}{\alpha}\right)\mathbb H(r).
			\label{eq:L1-drift-inline}
		\end{equation}
		Combining \eqref{eq:L1-prime-main}--\eqref{eq:L1-drift-inline}, and using the definitions of $\cG$ and $\cE$, we obtain
		\begin{equation}			\mathbb L_1'\geq \frac{2\alpha+n}{r}\mathbb L_1-C\cG(r)\mathbb D(r)-C\cE(r)\mathbb H(r).
			\label{eq:L1-prime-lower-main}
		\end{equation}
		
		We next differentiate $\mathbb L_2$. By
		Lemma~\ref{lem:weighted-differentiation},
		\begin{align}			\mathbb L_2'
			&=2\lambda r\int_{B_r}w\Delta w\,\omega_r^\alpha\dd
			+\lambda r^2\frac{d}{dr}
			\int_{B_r}w\Delta w\,\omega_r^\alpha\dd\notag\\
			&=\frac{2\alpha+n+2}{r}\mathbb L_2
			+\frac{\lambda r^2}{2(\alpha+1)r}\int_{B_r}\Delta(w\Delta w)\omega_r^{\alpha+1}\dd.
			\label{eq:L2-prime-main}
		\end{align}
		We use the expansion of $\Delta(w\Delta w)$ in Lemma~\ref{lem:integrand-expansion}.
		Since $\lambda r^2/(2(\alpha+1)r)\leq C\lambda r/\alpha$, it remains to bound the integral of each term in \eqref{eq:expansion-wDeltaw-lemma} against $\lambda r\omega_r^{\alpha+1}/\alpha$.
		
		By \eqref{eq:Deltaw-square-bound} and \eqref{eq:w2-alpha1-H},
		\begin{align}			\frac{\lambda r}{\alpha}\int_{B_r}(\Delta w)^2\omega_r^{\alpha+1}\dd
			&\leq C\frac{\lambda rW_0^2}{\alpha}\int_{B_r}w^2\omega_r^{\alpha+1}\dd
			+C\frac{\lambda rV_0^2}{\alpha}\int_{B_r}|\nabla u|^2\omega_r^{\alpha+1}\dd\notag\\
			&\leq C\frac{W_0^2r^3}{\alpha}\mathbb H(r)+C\frac{\lambda V_0^2r^3}{\alpha}\mathbb D(r).
			\label{eq:L2-Deltaw2-est}
		\end{align}
		The term $W^2w^2$ satisfies
		\begin{align}			\frac{\lambda r}{\alpha}\int_{B_r}W^2w^2\omega_r^{\alpha+1}\dd
			&\leq \frac{\lambda rW_0^2}{\alpha}\int_{B_r}w^2\omega_r^{\alpha+1}\dd\notag\\
			&\leq \frac{\lambda rW_0^2}{\alpha}\cdot\frac{r^2}{\lambda}\mathbb H(r)
			=\frac{W_0^2r^3}{\alpha}\mathbb H(r).
			\label{eq:L2-W2w2-est}
		\end{align}
		For the term $WwV^j\partial_ju$, \eqref{eq:prod-c} gives
		\begin{align}			\frac{\lambda r}{\alpha}\int_{B_r}|W|\,|w|\,|V|\,|\nabla u|\omega_r^{\alpha+1}\dd
			&\leq C\frac{\lambda rW_0V_0}{\alpha}\cdot\frac{r^2}{\sqrt\lambda}\sqrt{\mathbb H(r)\mathbb D(r)}\notag\\
			&=C\frac{W_0V_0\sqrt\lambda\,r^3}{\alpha}\sqrt{\mathbb H(r)\mathbb D(r)}\notag\\
			&\leq C\frac{W_0V_0\sqrt\lambda\,r^3}{\alpha}\mathbb D(r)
			+C\frac{W_0V_0\sqrt\lambda\,r^3}{\alpha}\mathbb H(r),
			\label{eq:L2-WV-est}
		\end{align}
		where the last line again uses $2\sqrt{\mathbb H\mathbb D}\leq \mathbb H+\mathbb D$.
		
		By \eqref{eq:prod-b},
		\begin{align}			\frac{\lambda r}{\alpha}\int_{B_r}|w|\,|\nabla W|\,|\nabla w|\omega_r^{\alpha+1}\dd
			&\leq C\frac{\lambda rW_1}{\alpha}\cdot\frac{r}{\lambda}\sqrt{\mathbb H(r)\mathbb D(r)}\notag\\
			&=C\frac{W_1r^2}{\alpha}\sqrt{\mathbb H(r)\mathbb D(r)}\notag\\
			&\leq C\frac{W_1r^2}{\alpha}\mathbb D(r)+C\frac{W_1r^2}{\alpha}\mathbb H(r).
			\label{eq:L2-gradW-est}
		\end{align}
		The term involving $\Delta W$ satisfies
		\begin{align}			\frac{\lambda r}{\alpha}\int_{B_r}|\Delta W|w^2\omega_r^{\alpha+1}\dd
			&\leq \frac{\lambda rW_2}{\alpha}\int_{B_r}w^2\omega_r^{\alpha+1}\dd\notag\\
			&\leq \frac{\lambda rW_2}{\alpha}\cdot\frac{r^2}{\lambda}\mathbb H(r)
			=\frac{W_2r^3}{\alpha}\mathbb H(r).
			\label{eq:L2-DeltaW-est}
		\end{align}
		For $w(\Delta V^j)\partial_ju$, the same product estimate \eqref{eq:prod-c} gives
		\begin{align}			\frac{\lambda r}{\alpha}\int_{B_r}|w|\,|\Delta V|\,|\nabla u|\omega_r^{\alpha+1}\dd
			&\leq C\frac{\lambda rV_2}{\alpha}\cdot\frac{r^2}{\sqrt\lambda}\sqrt{\mathbb H(r)\mathbb D(r)}\notag\\
			&=C\frac{\sqrt\lambda V_2r^3}{\alpha}\sqrt{\mathbb H(r)\mathbb D(r)}\notag\\
			&\leq C\frac{\sqrt\lambda V_2r^3}{\alpha}\mathbb D(r)+C\frac{\sqrt\lambda V_2r^3}{\alpha}\mathbb H(r).
			\label{eq:L2-DeltaV-est}
		\end{align}
		For $w\partial_iV^j\partial_{ij}u$, \eqref{eq:prod-d} yields
		\begin{align}			\frac{\lambda r}{\alpha}\int_{B_r}|w|\,|\nabla V|\,|D^2u|\omega_r^{\alpha+1}\dd
			&\leq C\sqrt\lambda V_1r^2\sqrt{\mathbb H(r)\mathbb D(r)}+C\frac{V_1r^3}{\alpha}\mathbb H(r)\notag\\
			&\leq C\sqrt\lambda V_1r^2\mathbb D(r)
			+C\left(\sqrt\lambda V_1r^2+\frac{V_1r^3}{\alpha}\right)\mathbb H(r).
			\label{eq:L2-gradV-Hessian-est}
		\end{align}
		The integration by parts estimate \eqref{eq:prod-h} gives
		\begin{equation}			\frac{\lambda r}{\alpha}\left|\int_{B_r}wV\cdot\nabla w\omega_r^{\alpha+1}\dd\right|
			\leq C\left(V_0r^2+\frac{V_1r^3}{\alpha}\right)\mathbb H(r).
			\label{eq:L2-wVgradw-est}
		\end{equation}
		For $W|\nabla w|^2$,
		\begin{align}			\frac{\lambda r}{\alpha}\int_{B_r}|W||\nabla w|^2\omega_r^{\alpha+1}\dd
			&\leq C\frac{\lambda rW_0}{\alpha}\,r^2\int_{B_r}|\nabla w|^2\omega_r^\alpha\dd\notag\\
			&\leq C\frac{\lambda rW_0}{\alpha}\,r^2\cdot\frac{\mathbb D(r)}{\lambda r^2}
			=C\frac{W_0r}{\alpha}\mathbb D(r).
			\label{eq:L2-W-gradw2-est}
		\end{align}
		For $\partial_iw\partial_iV^j\partial_ju$, use \eqref{eq:prod-a} with the coefficient $\lambda rV_1/\alpha$:
		\begin{align}			\frac{\lambda r}{\alpha}\int_{B_r}|\nabla w|\,|\nabla V|\,|\nabla u|\omega_r^{\alpha+1}\dd
			&\leq C\frac{\lambda rV_1}{\alpha}\cdot\frac{r}{\sqrt\lambda}\mathbb D(r)
			=C\frac{\sqrt\lambda V_1r^2}{\alpha}\mathbb D(r).
			\label{eq:L2-gradw-gradV-gradu-est}
		\end{align}
		Finally, \eqref{eq:prod-e} gives
		\begin{align*}			\frac{\lambda r}{\alpha}\int_{B_r}|\nabla w|\,|V|\,|D^2u|\omega_r^{\alpha+1}\dd
			&\leq C\sqrt\lambda V_0r\mathbb D(r)+C\frac{V_0r^2}{\alpha}\sqrt{\mathbb H(r)\mathbb D(r)}.
		\end{align*}
		Young's inequality gives
		\begin{equation}			C\frac{V_0r^2}{\alpha}\sqrt{\mathbb H(r)\mathbb D(r)}
			\leq C\sqrt\lambda V_0r\mathbb D(r)+C\frac{V_0r^3}{\alpha^2\sqrt\lambda}\mathbb H(r).
			\label{eq:final-young-inline}
		\end{equation}
		Consequently
		\begin{equation}			\frac{\lambda r}{\alpha}\int_{B_r}|\nabla w|\,|V|\,|D^2u|\omega_r^{\alpha+1}\dd
			\leq C\sqrt\lambda V_0r\mathbb D(r)+C\frac{V_0r^3}{\alpha^2\sqrt\lambda}\mathbb H(r).
			\label{eq:L2-gradw-V-Hessian-est}
		\end{equation}
		
		Combining \eqref{eq:L2-Deltaw2-est}--\eqref{eq:L2-gradw-V-Hessian-est} with \eqref{eq:L2-prime-main} gives
		\begin{align}			\mathbb L_2'(r)
			&\geq \frac{2\alpha+n+2}{r}\mathbb L_2(r)\notag\\
			&\quad-C\Bigg[
			\frac{\lambda V_0^2r^3}{\alpha}
			+\frac{W_0V_0\sqrt\lambda\,r^3}{\alpha}
			+\frac{W_1r^2}{\alpha}
			+\frac{\sqrt\lambda V_2r^3}{\alpha}
			\notag\\
			&\qquad\quad+\sqrt\lambda V_1r^2
			+\frac{W_0r}{\alpha}
			+\frac{\sqrt\lambda V_1r^2}{\alpha}
			+\sqrt\lambda V_0r
			\Bigg]\mathbb D(r)\notag\\
			&\quad-C\Bigg[
			\frac{W_0^2r^3}{\alpha}
			+\frac{W_0V_0\sqrt\lambda\,r^3}{\alpha}
			+\frac{W_1r^2}{\alpha}
			+\frac{W_2r^3}{\alpha}
			+\frac{\sqrt\lambda V_2r^3}{\alpha}
			\notag\\
			&\qquad\quad+\sqrt\lambda V_1r^2
			+\frac{V_1r^3}{\alpha}
			+V_0r^2
			+\frac{V_0r^3}{\alpha^2\sqrt\lambda}
			\Bigg]\mathbb H(r).
			\label{eq:L2-long-collected-bound}
		\end{align}
		Since $\alpha\geq2$, the definitions \eqref{eq:G-def} and \eqref{eq:E-def} bound the two brackets in \eqref{eq:L2-long-collected-bound} by $C\cG(r)$ and $C\cE(r)$, respectively. In particular, the mixed terms in the second bracket are included through $\cG(r)$ in \eqref{eq:E-def}.
		
		It follows that
		\begin{equation}			\mathbb L_2'\geq \frac{2\alpha+n+2}{r}\mathbb L_2-C\cG(r)\mathbb D(r)-C\cE(r)\mathbb H(r).
			\label{eq:L2-prime-lower-main}
		\end{equation}
		Adding \eqref{eq:L1-prime-lower-main} and \eqref{eq:L2-prime-lower-main} yields
		\begin{equation}			\mathbb L'\geq \frac{2\alpha+n}{r}\mathbb L+
			\frac{2}{r}\mathbb L_2-C\cG(r)\mathbb D-C\cE(r)\mathbb H.
			\label{eq:Lprime-lower-main}
		\end{equation}
		We also need $\mathbb L_2$ itself.  From $\Delta w=-Ww-V\cdot\nabla u$,
		\begin{align}			|\mathbb L_2|
			&\leq \lambda r^2W_0\int_{B_r}w^2\omega_r^\alpha\dd
			+\lambda r^2V_0\int_{B_r}|w|\,|\nabla u|\omega_r^\alpha\dd\notag\\
			&\leq W_0r^2\mathbb H(r)
			+\lambda r^2V_0
			\left(\int_{B_r}w^2\omega_r^{\alpha-1}\dd\right)^{1/2}
			\left(\int_{B_r}|\nabla u|^2\omega_r^{\alpha+1}\dd\right)^{1/2}\notag\\
			&\leq W_0r^2\mathbb H(r)+C V_0\left(\alpha\sqrt\lambda\,r^2+\lambda^{1/4}r^{5/2}\right)\mathbb H(r)\notag\\
			&\leq \frac14B(r)\mathbb H(r),
			\label{eq:L2-by-BH}
		\end{align}
		
		Since $\mathbb N=(\mathbb D+\mathbb L)/\mathbb H$, the quotient rule gives
		\begin{equation*}			\mathbb N'=\frac{(\mathbb D'+\mathbb L')\mathbb H-(\mathbb D+\mathbb L)\mathbb H'}{\mathbb H^2}=
			\frac{\mathbb D'\mathbb H-\mathbb D\mathbb H'}{\mathbb H^2}+\frac{\mathbb L'\mathbb H-\mathbb L\mathbb H'}{\mathbb H^2}.
		\end{equation*}
		By \eqref{eq:H-prime},
		\begin{align}			\frac{\mathbb L'\mathbb H-\mathbb L\mathbb H'}{\mathbb H^2}
			&=\frac{1}{\mathbb H}\left(\mathbb L'-\frac{2\alpha+n-2}{r}\mathbb L\right)-\frac{\mathbb L}{\mathbb H}\frac{\mathbb N}{\alpha r}.
			\label{eq:Lquot-start-main}
		\end{align}
		Using \eqref{eq:Lprime-lower-main}, \eqref{eq:L-by-BH}, and \eqref{eq:L2-by-BH},
		\begin{align}			\frac{1}{\mathbb H}\left(\mathbb L'-\frac{2\alpha+n-2}{r}\mathbb L\right)
			&\geq \frac{2}{r}\frac{\mathbb L+\mathbb L_2}{\mathbb H}-C\cG(r)\frac{\mathbb D}{\mathbb H}-C\cE(r)\notag\\
			&\geq -C\frac{B(r)}{r}-C\cG(r)\frac{\mathbb D}{\mathbb H}-C\cE(r).
			\label{eq:Lquot-middle-main}
		\end{align}
		Moreover,
		\begin{equation}			\frac{\mathbb D}{\mathbb H}=\mathbb N-\frac{\mathbb L}{\mathbb H}\leq \mathbb N+B(r),
			\label{eq:DoverH-for-L}
		\end{equation}
		and
		\begin{align}			-\frac{\mathbb L}{\mathbb H}\frac{\mathbb N}{\alpha r}
			&\geq -\frac{|\mathbb L|}{\mathbb H}\frac{|\mathbb N|}{\alpha r}
			\geq -\frac{B(r)}{\alpha r}\bigl(\mathbb N(r)+2B(r)\bigr),
			\label{eq:LN-term-main}
		\end{align}
		where the last step uses Lemma~\ref{lem:true-frequency-lower-bound},
		specifically \eqref{eq:true-frequency-lower-bound}.
		
		Suppose first that $\mathbb N(r)\geq0$. Substituting
		\eqref{eq:DoverH-for-L} and \eqref{eq:LN-term-main} into
		\eqref{eq:Lquot-start-main}--\eqref{eq:Lquot-middle-main}, we obtain
		\begin{align*}			\frac{\mathbb L'\mathbb H-\mathbb L\mathbb H'}{\mathbb H^2}
			&\geq -C\left(\cG(r)+\frac{B(r)}{\alpha r}\right)\mathbb N(r)\\
			&\quad-C\left[\cE(r)+\left(\cG(r)+\frac{2}{r}\right)B(r)+\frac{B(r)^2}{\alpha r}\right].
		\end{align*}
		Combining this with \eqref{eq:D-part-final} and enlarging the
		dimensional constant $C_0$ in \eqref{eq:Astar-def} and
		\eqref{eq:Pstar-def} gives $\mathbb N'+A\mathbb N\geq-P$ in this case.
		
		Assume now that $\mathbb N(r)<0$. We enlarge the same constant $C_0$, if
		necessary, so that the constants in \eqref{eq:Astar-def} and
		\eqref{eq:Pstar-def} apply in the calculation below. Using \eqref{eq:true-frequency-strong-bounds} in
		\eqref{eq:D-part-negative} and \eqref{eq:Lquot-start-main}, together
		with \eqref{eq:Lprime-lower-main}, yields
		\begin{equation*}
			\mathbb N'\geq-C\cG\left(\mathbb N+\frac{B}{4}\right)
			-\frac{B}{4\alpha r}\left(\mathbb N+\frac{B}{2}\right)
			-C\cE-\frac{B}{r},
		\end{equation*}
		where $C$ depends only on $n$. Taking $C_0\geq C$ and using $\mathbb N\geq-B/4$ once more, we obtain
		\begin{align*}
			\mathbb N'+A\mathbb N
			&\geq-C\cE-\frac{C_0}{4}\cG B-\frac{B}{r}
			-\left(\frac{C_0}{4}+\frac{1}{16}\right)\frac{B^2}{\alpha r}\\
			&\geq-P.
		\end{align*}
		Thus \eqref{eq:true-frequency-differential-ineq} holds in both
		cases, which proves the proposition.
	\end{proof}
	
	\begin{corollary}
		\label{cor:true-frequency-almost-monotone}
		For every fixed $r_0\in(0,R)$, the quantity
		\begin{equation}			\mathcal N(r):=
			e^{\int_{r_0}^rA(t)\,dt}\mathbb N(r)
			+\int_{r_0}^r e^{\int_{r_0}^sA(t)\,dt}P(s)\,ds
			\label{eq:monotone-frequency}
		\end{equation}
		is nondecreasing on every subinterval of $(r_0,R)$ on which $\mathbb H>0$.
	\end{corollary}
	
	\begin{proof}
		Differentiating \eqref{eq:monotone-frequency} gives
		\begin{align*}			\mathcal N'(r)
			&=e^{\int_{r_0}^rA(t)\,dt}\bigl(\mathbb N'(r)+A(r)\mathbb N(r)\bigr)
			+e^{\int_{r_0}^rA(t)\,dt}P(r)\\
			&=e^{\int_{r_0}^rA(t)\,dt}\bigl(\mathbb N'(r)+A(r)\mathbb N(r)+P(r)\bigr)\geq0,
		\end{align*}
		by Proposition~\ref{prop:true-frequency-differential-ineq}. Hence
		$\mathcal N$ is nondecreasing.
	\end{proof}
	
	\begin{lemma}
		\label{lem:comparison-almost-monotone}
		Let $0<a<b<R$ and assume $\mathbb H>0$ on $[a,b]$.  Define
		\begin{equation*}			\Phi(s,t):=\int_s^tA(\rho)\,d\rho,
			\qquad 0<a\leq s\leq t\leq b.
		\end{equation*}
		Then
		\begin{equation}			\mathbb N(a)\leq e^{\Phi(a,b)}\mathbb N(b)+\int_a^b e^{\Phi(a,s)}P(s)\,ds,
			\label{eq:N-comparison-upper}
		\end{equation}
		and
		\begin{equation}			\mathbb N(b)\geq e^{-\Phi(a,b)}\mathbb N(a)-\int_a^b e^{-\Phi(s,b)}P(s)\,ds.
			\label{eq:N-comparison-lower}
		\end{equation}
		Moreover,
		\begin{equation}			\frac{\mathbb H'(r)}{\mathbb H(r)}=\frac{2\alpha+n-2}{r}+\frac{\mathbb N(r)}{\alpha r}
			\label{eq:log-H-true-frequency}
		\end{equation}
		for every $r\in(a,b)$.
	\end{lemma}
	
	\begin{proof}
		The identity \eqref{eq:log-H-true-frequency} follows from
		\eqref{eq:H-prime} and $\mathbb N=(\mathbb D+\mathbb L)/\mathbb H$. By
		Proposition~\ref{prop:true-frequency-differential-ineq}, multiplying
		\eqref{eq:true-frequency-differential-ineq} by the integrating factor gives
		\begin{equation*}			\frac{d}{dr}\left(e^{\Phi(a,r)}\mathbb N(r)\right)
			=e^{\Phi(a,r)}\bigl(\mathbb N'(r)+A(r)\mathbb N(r)\bigr)
			\geq-e^{\Phi(a,r)}P(r).
		\end{equation*}
		Integration from $a$ to $b$ yields
		\[
		e^{\Phi(a,b)}\mathbb N(b)-\mathbb N(a)\geq-\int_a^b e^{\Phi(a,s)}P(s)\,ds,
		\]
		which is \eqref{eq:N-comparison-upper}. Multiplying by $e^{-\Phi(a,b)}$ and using $\Phi(a,b)-\Phi(a,s)=\Phi(s,b)$ gives \eqref{eq:N-comparison-lower}.
	\end{proof}
	
	By Corollary~\ref{cor:true-frequency-almost-monotone} and the comparison
	formulas above, we obtain a weighted three-ball inequality for $\mathbb H(r)$.
	Let $0<R_0<R$ be fixed. In what follows,
	the coefficient bounds and the functions $A,B,P$ are evaluated on
	$B_{R_0}$. Assume that $\mathbb M,\mathbb G\geq1$ and
	\begin{align*}		\|W\|_{C^2(B_{R_0})}&\leq \mathbb M,\\
		\|V\|_{C^2(B_{R_0})}&\leq \mathbb G.
	\end{align*}
	Unless otherwise stated, constants depend only on $n$ and $R_0$.
	
	We introduce the following quantities:
	\begin{align}		c_1&:=(2\alpha+n-2)\log\frac{2r_2}{r_1},&
		c_2&:=(2\alpha+n-2)\log\frac{r_3}{2r_2},\notag
		\\
		\beta_1&:=\frac1\alpha\int_{r_1}^{2r_2}
		e^{\Phi(s,2r_2)}\frac{ds}{s},&
		\gamma_1&:=\frac1\alpha\int_{2r_2}^{r_3}
		e^{-\Phi(2r_2,s)}\frac{ds}{s},
		\label{eq:beta-gamma-corrected}
	\end{align}
	and
	\begin{align}		R_1&:=\frac1\alpha\int_{r_1}^{2r_2}
		\left(\int_s^{2r_2}e^{\Phi(s,t)}P(t)\,dt\right)\frac{ds}{s},
		\label{eq:R1-free}\\
		R_2&:=\frac1\alpha\int_{2r_2}^{r_3}
		\left(\int_{2r_2}^{s}e^{-\Phi(t,s)}P(t)\,dt\right)\frac{ds}{s}.
		\label{eq:R2-free}
	\end{align}
	Let
	\begin{equation}\label{eq:kappa-corrected-H}
		\kappa:=\kappa(r_1,r_2,r_3):=\frac{\gamma_1}{\beta_1+\gamma_1},
	\end{equation}
	and define
	\begin{equation*}		\Lambda_{\mathbb H}:=\kappa c_1-(1-\kappa)c_2+\kappa R_1+(1-\kappa)R_2.
	\end{equation*}
	
	\begin{proposition}\label{prop:weighted-H-three-ball-corrected}
		Let
		\begin{equation*}			0<r_1<r_2<2r_2<r_3<R_0.
		\end{equation*}
		Then
		\begin{equation}\label{eq:H-three-ball-free}
			\mathbb H(2r_2)\leq e^{\Lambda_{\mathbb H}}\mathbb H(r_1)^\kappa \mathbb H(r_3)^{1-\kappa}.
		\end{equation}
	\end{proposition}
	
	\begin{proof}
		The conclusion is immediate if $u\equiv0$ in $B_{R_0}$. Otherwise,
		direct differentiation gives
		\begin{align*}
			\mathbb H'(r)
			&=2\lambda r\int_{B_r}w^2\omega_r^{\alpha-1}\,dx
			+2(\alpha-1)r\int_{B_r}
			\bigl(u^2+\lambda r^2w^2\bigr)\omega_r^{\alpha-2}\,dx\geq0.
		\end{align*}
		Thus $\mathbb H$ is continuous and nondecreasing. Put
		\[
		a:=\inf\{r\in(0,R_0):\mathbb H(r)>0\}.
		\]
		If $a>0$, fix $b\in(a,R_0)$. By
		\eqref{eq:N-comparison-upper}, $\mathbb N(s)$ is bounded above on $(a,b]$,
		since $A$ and $P$ are bounded on $[a,b]$. The identity
		\eqref{eq:log-H-true-frequency} then gives $\mathbb H'(s)/\mathbb H(s)\leq C$.
		Integration yields $\mathbb H(s)\geq e^{-C(b-s)}\mathbb H(b)$, contradicting
		$\mathbb H(a)=0$ as $s\downarrow a$. Hence $a=0$ and $\mathbb H(r)>0$ for
		every $0<r<R_0$.
		
		Integrating \eqref{eq:log-H-true-frequency} from $r_1$ to $2r_2$ gives
		\begin{equation}\label{eq:lower-interval-log-H-corrected}
			\log\frac{\mathbb H(2r_2)}{\mathbb H(r_1)}
			=c_1+\frac1\alpha\int_{r_1}^{2r_2}\mathbb N(s)\frac{ds}{s}.
		\end{equation}
		For $r_1\leq s\leq2r_2$, \eqref{eq:N-comparison-upper} implies
		\[
		\mathbb N(s)\leq e^{\Phi(s,2r_2)}\mathbb N(2r_2)
		+\int_s^{2r_2}e^{\Phi(s,t)}P(t)\,dt.
		\]
		Substitution into \eqref{eq:lower-interval-log-H-corrected} yields
		\begin{equation}\label{eq:lower-interval-log-H-short-corrected}
			\log\frac{\mathbb H(2r_2)}{\mathbb H(r_1)}
			\leq c_1+\beta_1\mathbb N(2r_2)+R_1.
		\end{equation}
		Similarly, integrating on $[2r_2,r_3]$ and using
		\eqref{eq:N-comparison-lower}, we obtain
		\begin{equation}\label{eq:upper-interval-log-H-short-corrected}
			\log\frac{\mathbb H(r_3)}{\mathbb H(2r_2)}
			\geq c_2+\gamma_1\mathbb N(2r_2)-R_2.
		\end{equation}
		Multiply \eqref{eq:lower-interval-log-H-short-corrected} by $\gamma_1$
		and \eqref{eq:upper-interval-log-H-short-corrected} by $\beta_1$.
		Eliminating $\mathbb N(2r_2)$ gives
		\begin{align*}			(\beta_1+\gamma_1)\log \mathbb H(2r_2)
			&\leq \gamma_1\log \mathbb H(r_1)+\beta_1\log \mathbb H(r_3)\\
			&\quad+\gamma_1c_1-\beta_1c_2+\gamma_1R_1+\beta_1R_2.
		\end{align*}
		Division by $\beta_1+\gamma_1$ and exponentiation prove
		\eqref{eq:H-three-ball-free}.
	\end{proof}
	
	We now choose $\alpha$ and $\lambda$ to optimize the preceding three-ball
	inequality, and recall we have defined $\mathbb K$ in \eqref{eq:intro-K}.
	\begin{lemma}\label{lem:optimized-weighted-threeball}
		Let $0<r_1<r_2<2r_2<r_3<R_0$. The parameters $\alpha$ and
		$\lambda$ can be chosen so that
		\begin{equation}\label{eq:H-three-ball-corrected}
			\mathbb H(2r_2)\leq
			\exp\left(C\mathbb K\log\frac{r_3}{2r_2}\right)
			\mathbb H(r_1)^\kappa \mathbb H(r_3)^{1-\kappa},
		\end{equation}
		where $C=C(n,R_0)>0$ and $\kappa$ is given by
		\eqref{eq:kappa-corrected-H} for the parameters in
		\eqref{eq:alpha-lambda-optimized-new}.
	\end{lemma}
	
	\begin{proof}
		Let $\alpha\geq2$ and $0<\lambda\leq1$. Define
		\begin{align}			F(\alpha,\lambda):={}&
			\alpha+1+\frac1{\alpha\sqrt\lambda}+\frac{\mathbb M}{\alpha}
			+\frac{\mathbb M^2\sqrt\lambda}{\alpha}
			+\mathbb G\sqrt\lambda+\frac{\mathbb G}{\alpha^2}
			+\frac{\mathbb G^2\lambda^{3/2}}\alpha .
			\label{eq:parameter-objective}
		\end{align}
		Recall we have defined $A(r)$ in \eqref{eq:Astar-def} and $P(r)$ in \eqref{eq:Pstar-def}. 
		We first prove
		\begin{equation}\label{eq:parameter-dependent-error}
			P(s)\leq C\alpha F(\alpha,\lambda)A(s),
			\qquad 0<s<R_0.
		\end{equation}
		The definition of $A$ gives
		\[
		A(s)\geq\frac{c}{\alpha\sqrt\lambda},\qquad
		A(s)\geq c\mathcal G(s),\qquad
		A(s)\geq\frac{cB(s)}{\alpha s}.
		\]
		Consequently, the definition of $P$ implies
		\begin{equation*}			\frac{P(s)}{\alpha A(s)}
			\leq C\left(1+\frac{B(s)}\alpha+\sqrt\lambda\,\mathcal E(s)\right).
		\end{equation*}
		Since $W_j\leq C\mathbb M$ and $V_j\leq C\mathbb G$ for $j=0,1,2$,
		\begin{align*}			\frac{B(s)}\alpha
			&\leq C\left(\frac1{\alpha\sqrt\lambda}+\frac{\mathbb M}{\alpha}
			+\mathbb G\sqrt\lambda+\frac{\mathbb G\lambda^{1/4}}\alpha\right),\\
			\sqrt\lambda\,\mathcal E(s)
			&\leq C\bigg(1+\frac1{\alpha\sqrt\lambda}+\frac{\mathbb M}{\alpha}
			+\frac{\mathbb M^2\sqrt\lambda}\alpha+\mathbb G\sqrt\lambda
			+\frac{\mathbb G}{\alpha^2}\\
			&\hspace{24mm}
			+\frac{\mathbb G^2\lambda^{3/2}}\alpha+\frac{\mathbb M\mathbb G\lambda}\alpha\bigg).
		\end{align*}
		Here we used $s<R_0$, $\mathbb M,\mathbb G\geq1$, $\alpha\geq2$, and $\lambda\leq1$.
		The remaining mixed terms satisfy
		\begin{align*}			\frac{\mathbb G\lambda^{1/4}}\alpha
			&\leq\frac12\left(\mathbb G\sqrt\lambda+\frac{\mathbb G}{\alpha^2}\right),\\
			\frac{\mathbb M\mathbb G\lambda}\alpha
			&\leq\frac12\left(\frac{\mathbb M^2\sqrt\lambda}\alpha
			+\frac{\mathbb G^2\lambda^{3/2}}\alpha\right).
		\end{align*}
		This proves \eqref{eq:parameter-dependent-error}.
		
		We next estimate the terms in $\Lambda_{\mathbb H}$. Since $A\geq0$,
		\[
		\beta_1\geq\frac1\alpha\log\frac{2r_2}{r_1},
		\qquad
		\gamma_1\leq\frac1\alpha\log\frac{r_3}{2r_2}.
		\]
		It follows that
		\begin{equation}\label{eq:cancellation-no-r1-corrected}
			\kappa c_1-(1-\kappa)c_2
			=\frac{2\alpha+n-2}{\beta_1+\gamma_1}
			\left(\gamma_1\log\frac{2r_2}{r_1}
			-\beta_1\log\frac{r_3}{2r_2}\right)\leq0.
		\end{equation}
		Changing the order of integration in \eqref{eq:R1-free} and using
		$\Phi(s,t)=\Phi(s,2r_2)-\Phi(t,2r_2)$, we obtain
		\begin{align*}			R_1
			&=\frac1\alpha\int_{r_1}^{2r_2}P(t)e^{-\Phi(t,2r_2)}
			\left(\int_{r_1}^t e^{\Phi(s,2r_2)}\frac{ds}{s}\right)dt\\
			&\leq\beta_1\int_{r_1}^{2r_2}P(t)e^{-\Phi(t,2r_2)}\,dt\\
			&\leq C\alpha F(\alpha,\lambda)\beta_1
			\int_{r_1}^{2r_2}A(t)e^{-\Phi(t,2r_2)}\,dt\\
			&\leq C\alpha F(\alpha,\lambda)\beta_1.
		\end{align*}
		In the last step we used
		$\partial_t e^{-\Phi(t,2r_2)}=A(t)e^{-\Phi(t,2r_2)}$.
		Since $\kappa\beta_1\leq\gamma_1$,
		\begin{equation}\label{eq:kappa-R1-parameter-estimate}
			\kappa R_1\leq C F(\alpha,\lambda)\log\frac{r_3}{2r_2}.
		\end{equation}
		Similarly, changing the order of integration in \eqref{eq:R2-free} and
		using \eqref{eq:parameter-dependent-error},
		\begin{align}			R_2
			&=\frac1\alpha\int_{2r_2}^{r_3}P(t)
			\left(\int_t^{r_3}e^{-\Phi(t,s)}\frac{ds}{s}\right)dt\notag\\
			&\leq C F(\alpha,\lambda)
			\int_{2r_2}^{r_3}\left(\int_{2r_2}^s
			A(t)e^{-\Phi(t,s)}\,dt\right)\frac{ds}{s}\notag\\
			&=C F(\alpha,\lambda)
			\int_{2r_2}^{r_3}\left(1-e^{-\Phi(2r_2,s)}\right)\frac{ds}{s}\notag\\
			&\leq C F(\alpha,\lambda)\log\frac{r_3}{2r_2}.
			\label{eq:R2-parameter-estimate}
		\end{align}
		Combining \eqref{eq:cancellation-no-r1-corrected},
		\eqref{eq:kappa-R1-parameter-estimate}, and
		\eqref{eq:R2-parameter-estimate} gives the single exponent estimate
		\begin{equation}\label{eq:Lambda-H-parameter-bound}
			\Lambda_{\mathbb H}\leq C F(\alpha,\lambda)\log\frac{r_3}{2r_2}.
		\end{equation}
		
		It remains to choose the parameters. The lower bound of	\eqref{eq:parameter-objective} is determined by
		\[
		F(\alpha,\lambda)\geq1,
		\qquad
		F(\alpha,\lambda)\geq\alpha+\frac{\mathbb M}{\alpha}\geq2\mathbb M^{1/2},
		\qquad
		F(\alpha,\lambda)\geq\alpha+\frac{\mathbb G}{\alpha^2}\geq c\mathbb G^{1/3}.
		\]
		Thus $F(\alpha,\lambda)\geq c\mathbb K$. The terms
		$\alpha$ and $1/(\alpha\sqrt\lambda)$ are balanced when
		$\lambda$ has order $\alpha^{-4}$. Now, choose
		\begin{equation}\label{eq:alpha-lambda-optimized-new}
			\alpha:=\lceil \mathbb K\rceil,\qquad \lambda:=\mathbb K^{-4}.
		\end{equation}
		Since $\mathbb K\geq2$, this choice satisfies $\alpha\geq2$ and
		$0<\lambda\leq1$. Moreover,
		\begin{align}			F(\lceil \mathbb K\rceil,\mathbb K^{-4})
			&\leq C\left(\mathbb K+\frac{\mathbb M}{\mathbb K}+\frac{\mathbb M^2}{\mathbb K^3}
			+\frac{\mathbb G}{\mathbb K^2}+\frac{\mathbb G^2}{\mathbb K^7}\right)\notag\\
			&\leq C\mathbb K,
			\label{eq:optimized-F-bound}
		\end{align}
		because $\mathbb M\leq \mathbb K^2$ and $\mathbb G\leq \mathbb K^3$. Hence the chosen parameters
		minimize the bound up to a constant. Substitution into
		\eqref{eq:Lambda-H-parameter-bound} and then
		\eqref{eq:H-three-ball-free} proves
		\eqref{eq:H-three-ball-corrected}.
	\end{proof}
	
	For the parameters \eqref{eq:alpha-lambda-optimized-new}, the definitions
	give
	\begin{align}		B(s)&\leq C\left[\mathbb K^2s+\mathbb K^2s^2+
		\mathbb K^3\left(\alpha \mathbb K^{-2}s^2+\mathbb K^{-1}s^{5/2}\right)\right]
		\leq C\alpha \mathbb K s,\notag\\
		c\mathbb K\leq A(s)&\leq C(\mathbb K+\mathbb K^2s),\qquad
		P(s)\leq C\alpha \mathbb K A(s),\qquad 0<s<R_0.
		\label{eq:Astar-lower-upper-optimized}
	\end{align}
	Indeed, $A\geq c/(\alpha\sqrt\lambda)\geq c\mathbb K$. In $\mathcal G$,
	the terms $s/\sqrt\lambda$ and
	$W_0V_0\sqrt\lambda\,s^3/\alpha$ are bounded by $C\mathbb K^2s$, and all
	the remaining terms are bounded by $C(\mathbb K+\mathbb K^2s)$. Also
	$B/(\alpha s)\leq C\mathbb K$, while the bound for $P$ follows from
	\eqref{eq:parameter-dependent-error} and \eqref{eq:optimized-F-bound}.
	In particular,
	\begin{equation}\label{eq:natural-scale-A-P}
		c\mathbb K\leq A(s)\leq C\mathbb K,\qquad
		P(s)\leq C\alpha \mathbb K^2,\qquad
		0<s<R_0,\quad s\leq \mathbb K^{-1}.
	\end{equation}
	
	We show the optimized unweighted three-ball inequality for
	\begin{equation*}		h(r):=\int_{B_r}\left(u^2+\lambda r^2w^2\right)\dd.
	\end{equation*}
	
	\begin{lemma}\label{lem:h-three-ball-corrected}
		Let $0<r_1<r_2<2r_2<r_3<R_0$ and let $\kappa$ be given by
		\eqref{eq:kappa-corrected-H}, with the parameters chosen in
		\eqref{eq:alpha-lambda-optimized-new}. Then
		\begin{equation}\label{eq:h-three-ball-corrected}
			h(r_2)\leq
			\exp\left(C\mathbb K\log\frac{r_3}{r_2}\right)
			h(r_1)^\kappa h(r_3)^{1-\kappa}.
		\end{equation}
	\end{lemma}
	
	\begin{proof}
		We first compare $\mathbb H$ and $h$ for arbitrary $\alpha\geq2$ and
		$\lambda>0$. Since
		$\omega_r\leq r^2$ in $B_r$ and $\omega_{2r_2}\geq3r_2^2$ in $B_{r_2}$,
		\begin{equation*}			\mathbb H(r)\leq r^{2\alpha-2}h(r),\qquad
			\mathbb H(2r_2)\geq(3r_2^2)^{\alpha-1}h(r_2).
		\end{equation*}
		Applying these inequalities to \eqref{eq:H-three-ball-free} yields
		\begin{align*}			h(r_2)
			&\leq(3r_2^2)^{1-\alpha}e^{\Lambda_{\mathbb H}}
			\mathbb H(r_1)^\kappa \mathbb H(r_3)^{1-\kappa}\\
			&\leq(3r_2^2)^{1-\alpha}e^{\Lambda_{\mathbb H}}
			\bigl(r_1^{2\alpha-2}h(r_1)\bigr)^\kappa
			\bigl(r_3^{2\alpha-2}h(r_3)\bigr)^{1-\kappa}\\
			&=e^{\Lambda_h}h(r_1)^\kappa h(r_3)^{1-\kappa},
		\end{align*}
		where
		\begin{align*}			\Lambda_h:={}&\Lambda_{\mathbb H}+(2\alpha-2)
			\bigl(\kappa\log r_1+(1-\kappa)\log r_3-\log r_2\bigr)
			+(1-\alpha)\log3.
		\end{align*}
		As $r_1<r_2<r_3$ and $0<\kappa<1$,
		\[
		\kappa\log r_1+(1-\kappa)\log r_3-\log r_2
		=\kappa\log\frac{r_1}{r_2}+(1-\kappa)\log\frac{r_3}{r_2}
		\leq\log\frac{r_3}{r_2}.
		\]
		The bounds \eqref{eq:Lambda-H-parameter-bound} and
		\eqref{eq:optimized-F-bound}, together with $\alpha\leq2\mathbb K$ and
		$(1-\alpha)\log3\leq0$, give
		\begin{equation*}			\Lambda_h\leq C\mathbb K\log\frac{r_3}{r_2},
		\end{equation*}
		proving \eqref{eq:h-three-ball-corrected}.
	\end{proof}
	
	We are ready to prove the $L^\infty$ three-ball inequality.
	We use the parameters in \eqref{eq:alpha-lambda-optimized-new}.
	For $0<\rho<R_0$ define
	\begin{equation}		\mathcal C(\rho):=1+\rho^2\|W\|_{L^\infty(B_\rho)}+\rho^3\|V\|_{L^\infty(B_\rho)}.
		\label{eq:C-rho-def}
	\end{equation}
	
	The required local $L^\infty$--$L^2$ and $W^{2,2}$--$L^\infty$
	estimates are established in Lemmas~\ref{lem:local-Linfty-L2-new}
	and~\ref{lem:local-W22-Linfty-new} in the Appendix. We apply them below
	to remove the Laplacian from the three-ball inequality for $h$.
	
	For $0<r<\rho<R_0$ define
	\begin{equation}		\mathcal B(r,\rho)
		:=r^{n/2}+\mathbb K^{-2}r\,\mathcal C(\rho)(\rho-r)^{-2}\rho^{n/2}.
		\label{eq:Bmu-def}
	\end{equation}
	
	\begin{samepage}
		\begin{proposition}
			\label{prop:Linfty-threeball-corrected}
			Assume
			\begin{equation*}				0<r_1<r_2<4r_2<r_3<R_0.
			\end{equation*}
			Let
			\begin{equation*}				\kappa_1:=\kappa(r_1,2r_2,r_3),
			\end{equation*}
			where $\kappa(r_1,r_2,r_3)$ is defined in \eqref{eq:kappa-corrected-H}.  If $r_1<\rho_1<2r_2$ and $r_3<\rho_3<R_0$, then
			\begin{align}				\|u\|_{L^\infty(B_{r_2})}
				&\leq C\mathcal C(2r_2)^{n/4} r_2^{-n/2}
				\exp\left(C\mathbb K\log\frac{r_3}{2r_2}\right)
				\notag\\
				&\quad\times
				\mathcal B(r_1,\rho_1)^{\kappa_1}
				\mathcal B(r_3,\rho_3)^{1-\kappa_1}
				\|u\|_{L^\infty(B_{\rho_1})}^{\kappa_1}
				\|u\|_{L^\infty(B_{\rho_3})}^{1-\kappa_1}.
				\label{eq:Linfty-threeball-corrected}
			\end{align}
		\end{proposition}
	\end{samepage}
	
	\begin{proof}
		By Lemma~\ref{lem:local-Linfty-L2-new} and the definition of $h$,
		\begin{equation*}			\|u\|_{L^\infty(B_{r_2})}
			\leq C\mathcal C(2r_2)^{n/4}r_2^{-n/2}
			\|u\|_{L^2(B_{2r_2})}
			\leq C\mathcal C(2r_2)^{n/4}r_2^{-n/2}h(2r_2)^{1/2}.
		\end{equation*}
		Applying Lemma~\ref{lem:h-three-ball-corrected} with middle
		radius $2r_2$, we obtain
		\begin{align}			\|u\|_{L^\infty(B_{r_2})}
			&\leq C\mathcal C(2r_2)^{n/4}r_2^{-n/2}
			\exp\left(C\mathbb K\log\frac{r_3}{2r_2}\right)
			h(r_1)^{\kappa_1/2}h(r_3)^{(1-\kappa_1)/2}.
			\label{eq:Linfty-middle-before-endpoints}
		\end{align}
		For $0<r<\rho<R_0$, Lemma~\ref{lem:local-W22-Linfty-new} gives
		\begin{align*}			h(r)^{1/2}
			&\leq \|u\|_{L^2(B_r)}+\mathbb K^{-2}r\|\Delta u\|_{L^2(B_r)}\\
			&\leq C\left[r^{n/2}
			+\mathbb K^{-2}r\mathcal C(\rho)(\rho-r)^{-2}\rho^{n/2}\right]
			\|u\|_{L^\infty(B_\rho)}\\
			&=C\mathcal B(r,\rho)\|u\|_{L^\infty(B_\rho)}.
		\end{align*}
		Substituting this inequality at $(r_1,\rho_1)$ and $(r_3,\rho_3)$
		into \eqref{eq:Linfty-middle-before-endpoints} proves
		\eqref{eq:Linfty-threeball-corrected}.
	\end{proof}
	
	We prove Theorem~\ref{thm:main-threeball} by applying the preceding
	proposition at fixed radii after rescaling. This uses the same parameter
	choice \eqref{eq:alpha-lambda-optimized-new}; the coefficient bounds change
	under the rescaling.
	
	\begin{proof}[Proof of Theorem~\ref{thm:main-threeball}]
		Fix $0<r\leq R/40$ and, for $y\in B_{40}$, define
		\[
		\widetilde u(y):=u(ry),\qquad
		\widetilde W(y):=r^2W(ry),\qquad
		\widetilde V(y):=r^3V(ry).
		\]
		Then
		\begin{equation}\label{eq:scaled-main-equation}
			\Delta^2\widetilde u+\widetilde W\Delta\widetilde u
			+\widetilde V\cdot\nabla\widetilde u=0
			\qquad\text{in }B_{40}.
		\end{equation}
		Set
		\begin{equation*}			\mathbb K_r:=2+
			\max\left\{1,\|\widetilde W\|_{C^2(B_{35})}\right\}^{1/2}
			+\max\left\{1,\|\widetilde V\|_{C^2(B_{35})}\right\}^{1/3}.
		\end{equation*}
		The definitions of $\widetilde W$ and $\widetilde V$ give
		\begin{align*}			\|\widetilde W\|_{C^2(B_{35})}
			&\leq (r^2+r^3+r^4)\mathbb M\leq C(n,R)r^2\mathbb M,\\
			\|\widetilde V\|_{C^2(B_{35})}
			&\leq (r^3+r^4+r^5)\mathbb G\leq C(n,R)r^3\mathbb G,
		\end{align*}
		and hence
		\begin{equation}\label{eq:scaled-main-K-bound}
			\mathbb K_r\leq C(n,R)\left(1+r\mathbb M^{1/2}+r\mathbb G^{1/3}\right)
			\leq C(n,R)\mathbb K.
		\end{equation}
		
		Apply Proposition~\ref{prop:Linfty-threeball-corrected} to
		\eqref{eq:scaled-main-equation} with $R_0=35$ and
		\[
		r_1=\frac12,\qquad r_2=3,\qquad r_3=15,
		\qquad \rho_1=1,\qquad \rho_3=30.
		\]
		Here its parameters are $\alpha=\lceil \mathbb K_r\rceil$ and
		$\lambda=\mathbb K_r^{-4}$. In \eqref{eq:C-rho-def} and
		\eqref{eq:Bmu-def}, the coefficient bounds and $\mathbb K$ are therefore
		replaced by those of the scaled equation and by $\mathbb K_r$. Since
		\[
		\max\left\{1,\|\widetilde W\|_{C^2(B_{35})}\right\}\leq \mathbb K_r^2,
		\qquad
		\max\left\{1,\|\widetilde V\|_{C^2(B_{35})}\right\}\leq \mathbb K_r^3,
		\]
		the fixed radii give
		\[
		\mathcal C(6)+\mathcal C(30)\leq C\mathbb K_r^3,
		\qquad
		\mathcal B(1/2,1)+\mathcal B(15,30)\leq C\mathbb K_r^C.
		\]
		The polynomial factors are absorbed into the exponential, so the
		proposition gives
		\begin{equation}\label{eq:scaled-main-threeball}
			\|\widetilde u\|_{L^\infty(B_3)}
			\leq e^{C\mathbb K_r}
			\|\widetilde u\|_{L^\infty(B_1)}^{\kappa_r}
			\|\widetilde u\|_{L^\infty(B_{30})}^{1-\kappa_r},
		\end{equation}
		where $\kappa_r=\kappa(1/2,6,15)$ is computed for the scaled
		coefficients and the parameters above.
		
		By \eqref{eq:Astar-lower-upper-optimized},
		$A(s)\leq C(\mathbb K_r+\mathbb K_r^2s)$ for the scaled equation. Thus
		\[
		0\leq\Phi(s,t)\leq C\mathbb K_r^2,
		\qquad \frac12\leq s\leq t\leq15.
		\]
		The definitions \eqref{eq:beta-gamma-corrected} therefore yield
		\begin{align*}
			\beta_1
			&\leq\frac{e^{C\mathbb K_r^2}}\alpha\log\frac{12}{1/2}
			=\frac{e^{C\mathbb K_r^2}}\alpha\log24,\\
			\gamma_1
			&\geq\frac{e^{-C\mathbb K_r^2}}\alpha\log\frac{15}{12}
			=\frac{e^{-C\mathbb K_r^2}}\alpha\log\frac54.
		\end{align*}
		Consequently, $\kappa_r\geq c e^{-C\mathbb K_r^2}$. In view of
		\eqref{eq:scaled-main-K-bound}, after reducing $c$ if necessary, we may
		choose
		\[
		\theta_0:=\min\left\{\frac12,c\exp\bigl(-C(n,R)\mathbb K^2\bigr)\right\}.
		\]
		Since
		$\|\widetilde u\|_{L^\infty(B_1)}
		\leq\|\widetilde u\|_{L^\infty(B_{30})}$ and
		$\kappa_r\geq\theta_0$, replacing $\kappa_r$ by $\theta_0$ on the
		right-hand side of \eqref{eq:scaled-main-threeball} only enlarges it.
		Scaling back and using \eqref{eq:scaled-main-K-bound} proves
		\eqref{eq:main-threeball}.
	\end{proof}
	
	\begin{remark}
		If $0<r\leq\min\{R/40,\mathbb K^{-1}\}$, then
		\eqref{eq:scaled-main-K-bound} gives $\mathbb K_r\leq C(n,R)$. The same proof
		therefore yields
		\[
		\|u\|_{L^\infty(B_{3r})}
		\leq C(n,R)\|u\|_{L^\infty(B_r)}^{\theta_0}
		\|u\|_{L^\infty(B_{30r})}^{1-\theta_0},
		\]
		where $\theta_0$ depends only on $n$ and $R$.
	\end{remark}
	
	\section{Vanishing order of solutions}
	\label{sec:vanishing}
	
	In this section, we prove Theorem~\ref{thm:quantitative-vanishing-order-final}.
	We apply the three-ball inequality on balls of radius comparable to
	$\mathbb K^{-1}$ and estimate its interpolation exponent as the inner radius
	tends to zero.
	
	\begin{proof}[Proof of Theorem~\ref{thm:quantitative-vanishing-order-final}]
		We take $R_0=1$ in Section~\ref{sec:threeball}. The coefficient bounds
		\eqref{eq:intro-coefficient-bounds} hold on $B_1$, so that we may use
		$\mathbb K$ from \eqref{eq:intro-K}, with $\alpha=\lceil \mathbb K\rceil$ and
		$\lambda=\mathbb K^{-4}$. Fix $0<r\leq r_*$ and choose
		\begin{equation*}			r_1=\frac r2,\qquad r_2=r_*,\qquad r_3=8r_*,
			\qquad \rho_1=r,\qquad \rho_3=9r_*.
		\end{equation*}
		These radii satisfy the hypotheses of
		Proposition~\ref{prop:Linfty-threeball-corrected}. Its interpolation exponent is
		\begin{equation*}			\kappa_1=\kappa(r/2,2r_*,8r_*).
		\end{equation*}
		By \eqref{eq:natural-scale-A-P}, $A(s)\leq C\mathbb K$ on $[r/2,8r_*]$.
		Consequently,
		\begin{equation*}			0\leq\int_s^t A(\tau)\,d\tau\leq C\mathbb K(t-s)\leq C,
			\qquad r/2\leq s\leq t\leq8r_*.
		\end{equation*}
		The exponential factors in the definition of $\kappa_1$ are thus bounded
		above and below by positive dimensional constants. In particular,
		\begin{align*}			\int_{4r_*}^{8r_*}\exp\left(-\int_{4r_*}^s A(t)\,dt\right)\frac{ds}{s}
			&\geq c\log2,\\
			\int_{r/2}^{4r_*}\exp\left(\int_s^{4r_*} A(t)\,dt\right)\frac{ds}{s}
			&\leq C\log\frac{8r_*}{r}.
		\end{align*}
		It follows from \eqref{eq:kappa-corrected-H} that
		\begin{equation}\label{eq:kappa-inverse-final}
			\kappa_1\geq\frac{c}{\log(16r_*/r)},
			\qquad \frac1{\kappa_1}\leq C\log\frac{16r_*}{r}.
		\end{equation}
		
		We next estimate the other factors in
		\eqref{eq:Linfty-threeball-corrected}. Since $r_3/(2r_2)=4$ and
		$r_*^{-n/2}=(100\mathbb K)^{n/2}$, their contribution is bounded by $\exp(C\mathbb K)$.
		Moreover, \eqref{eq:intro-coefficient-bounds} gives
		\[
		\rho^2\|W\|_{L^\infty(B_\rho)}
		+\rho^3\|V\|_{L^\infty(B_\rho)}\leq C,
		\qquad 0<\rho\leq9r_*.
		\]
		Hence $\mathcal C(\rho)\leq C$ on this interval. For the inner endpoint,
		\begin{align}			\mathcal B(r/2,r)
			&=2^{-n/2}r^{n/2}
			+2\mathbb K^{-2}\mathcal C(r) r^{n/2-1}
			\leq C,\label{eq:Bmu-small-endpoint-bound}
		\end{align}
		where we used $n\geq2$ and $r\leq r_*<1$. Similarly,
		\begin{equation}\label{eq:Bmu-outer-endpoint-bound}
			\mathcal B(8r_*,9r_*)
			=(8r_*)^{n/2}
			+8\mathbb K^{-2}\mathcal C(9r_*) r_*^{-1}(9r_*)^{n/2}
			\leq C.
		\end{equation}
		Proposition~\ref{prop:Linfty-threeball-corrected}, together with
		\eqref{eq:normalization-final}, now yields
		\begin{equation*}			1\leq\|u\|_{L^\infty(B_{r_*})}
			\leq\exp(C\mathbb K)\|u\|_{L^\infty(B_r)}^{\kappa_1}
			\mathcal M^{1-\kappa_1}.
		\end{equation*}
		In particular, $\|u\|_{L^\infty(B_r)}>0$. Taking logarithms and using
		\eqref{eq:kappa-inverse-final}, we obtain
		\begin{equation}\label{eq:quantitative-vanishing-final}
			\|u\|_{L^\infty(B_r)}
			\geq\exp\left[-C(\mathbb K+\log\mathcal M)\log\frac{16r_*}{r}\right],
			\qquad 0<r\leq r_*.
		\end{equation}
		Taking logarithms and dividing by $\log r<0$ reverses the inequality:
		\[
		\frac{\log\|u\|_{L^\infty(B_r)}}{\log r}
		\leq-C(\mathbb K+\log\mathcal M)
		\frac{\log(16r_*/r)}{\log r}.
		\]
		Since $\log(16r_*/r)/\log r\to-1$, taking the limsup gives the
		asserted bound at the origin; here $\mathbb M,\mathbb G\geq1$ and
		$\mathbb K\leq2(\mathbb M^{1/2}+\mathbb G^{1/3})$.
		
		Fix $x_0\in B_{r_*/2}$ and apply
		Proposition~\ref{prop:Linfty-threeball-corrected} to the translated
		function $y\mapsto u(x_0+y)$ with $R_0=1$ and
		\[
		r_1=\frac r2,\qquad r_2=2r_*,\qquad r_3=16r_*,
		\qquad \rho_1=r,\qquad \rho_3=18r_*.
		\]
		For $0<r\leq r_*$ these radii satisfy all the strict hypotheses of
		the proposition. Moreover,
		\[
		B_{r_*}(0)\subset B_{2r_*}(x_0),\qquad
		B_{18r_*}(x_0)\subset B_{(37/2)r_*}(0)\subset B_{10}(0),
		\]
		and $B_1(x_0)\subset B_R(0)$ because $R\geq20$. Hence the middle
		norm is at least one, the outer norm is at most $\mathcal M$, and the
		coefficient bounds used above remain valid after translation.
		
		Since $16r_*<\mathbb K^{-1}$, \eqref{eq:natural-scale-A-P} gives
		$A(s)\leq C\mathbb K$ for the translated coefficients on
		$[r/2,16r_*]$. Hence the exponential factors in the definition of
		$\kappa$ are bounded above and below by dimensional constants. If
		\[
		\kappa_{x_0}:=\kappa(r/2,4r_*,16r_*),
		\]
		then
		\[
		\gamma_1\geq\frac{c}{\alpha}\log2,
		\qquad
		\beta_1\leq\frac{C}{\alpha}\log\frac{16r_*}{r},
		\]
		and therefore
		\[
		\frac1{\kappa_{x_0}}\leq C\log\frac{32r_*}{r}.
		\]
		The endpoint estimates used in
		\eqref{eq:Bmu-small-endpoint-bound}--\eqref{eq:Bmu-outer-endpoint-bound}
		also give
		\[
		\mathcal B(r/2,r)+\mathcal B(16r_*,18r_*)\leq C.
		\]
		The polynomial middle-radius factor is absorbed into $e^{C\mathbb K}$, so
		Proposition~\ref{prop:Linfty-threeball-corrected} and the ball
		containments above yield
		\[
		1\leq\|u\|_{L^\infty(B_{2r_*}(x_0))}
		\leq e^{C\mathbb K}\|u\|_{L^\infty(B_r(x_0))}^{\kappa_{x_0}}
		\mathcal M^{1-\kappa_{x_0}}.
		\]
		Solving for the inner norm gives
		\[
		\|u\|_{L^\infty(B_r(x_0))}
		\geq \exp\left[-C(\mathbb K+\log\mathcal M)\log\frac{32r_*}{r}\right],
		\qquad 0<r\leq r_*.
		\]
		Dividing by $\log r<0$ and taking the limsup proves the stated
		pointwise estimate for every $x_0\in B_{r_*/2}$.
	\end{proof}
	
	\begin{remark}
		More generally, suppose that
		$\|u\|_{L^\infty(B_{10})}\leq \mathcal M_0$ and
		$\|u\|_{L^\infty(B_{r_*})}\geq m_*>0$.
		Applying \eqref{eq:quantitative-vanishing-final} to $u/m_*$ gives
		\[
		\|u\|_{L^\infty(B_r)}
		\geq m_*\exp\left[-C\left(\mathbb K+\log\frac{\mathcal M_0}{m_*}\right)
		\log\frac{16r_*}{r}\right],\qquad 0<r\leq r_*.
		\]
		Thus the dependence on the normalization is through the ratio
		$\mathcal M_0/m_*$.
	\end{remark}
	
	\appendix
	\section{Auxiliary identities and local regularity estimates}
	\label{sec:appendix}
	
	We first establish the weighted differentiation identities used throughout the paper.
	
	\begin{lemma}
		\label{lem:weighted-differentiation}
		Let $\gamma\geq1$, let $f\in C^2(B_R)$, and define
		\begin{equation*}			F(r):=\int_{B_r}f(x)\omega_r(x)^\gamma\dd .
		\end{equation*}
		Then, for every $0<r<R$,
		\begin{equation*}			F'(r)=\frac{2\gamma+n}{r}F(r)+\frac{1}{r}\int_{B_r}x\cdot\nabla f(x)\omega_r^\gamma\dd,
		\end{equation*}
		and also
		\begin{equation*}			F'(r)=\frac{2\gamma+n}{r}F(r)+\frac{1}{2(\gamma+1)r}\int_{B_r}\Delta f(x)\omega_r^{\gamma+1}\dd .
		\end{equation*}
	\end{lemma}
	
	\begin{proof}
		See \cite[Lemma 3.1]{Davey2025} or \cite[Lemma 1]{TengWangZhuBilaplace2026}.
	\end{proof}
	
	We now establish the local regularity estimates used in the proof of the
	$L^\infty$ three-ball inequality. Throughout this part, $u$ solves
	\eqref{eq:main-equation}, $0<R_0<R$, and $\mathcal C(\rho)$ is defined in
	\eqref{eq:C-rho-def}. We just state the results used here and the proof can be found in \cite[Lemma 17]{TengWangZhuBilaplace2026}.
	
	\begin{lemma}\label{lem:W4p-lower-order}
		Let $1<p<\infty$ and $0<r<\rho<R_0$. For $j=1,2,3,4$, one has
		\begin{equation*}			\|D^ju\|_{L^p(B_r)}
			\leq C(n,p,R_0)\mathcal C(\rho)^{j/2}(\rho-r)^{-j}
			\|u\|_{L^p(B_\rho)}.
		\end{equation*}
		In particular,
		\begin{equation*}			\|u\|_{W^{4,p}(B_r)}
			\leq C(n,p,R_0)\mathcal C(\rho)^2(\rho-r)^{-4}
			\|u\|_{L^p(B_\rho)}.
		\end{equation*}
	\end{lemma}

	\begin{lemma}\label{lem:local-Linfty-L2-new}
		For $0<r<\rho<R_0$,
		\begin{equation*}			\|u\|_{L^\infty(B_r)}
			\leq C(n,R_0)\mathcal C(\rho)^{n/4}(\rho-r)^{-n/2}
			\|u\|_{L^2(B_\rho)}.
		\end{equation*}
	\end{lemma}
	
	\begin{lemma}\label{lem:local-W22-Linfty-new}
		For $0<r<\rho<R_0$,
		\begin{equation*}			\|u\|_{W^{2,2}(B_r)}
			\leq C(n,R_0)\mathcal C(\rho)(\rho-r)^{-2}\rho^{n/2}
			\|u\|_{L^\infty(B_\rho)}.
		\end{equation*}
	\end{lemma}

\end{document}